\documentclass[10pt]{amsart}
\PassOptionsToPackage{shortlabels}{enumitem}
\usepackage{amssymb, enumitem}
\usepackage[all]{xy}
\usepackage{hyperref, aliascnt}
\usepackage[english]{babel}
\usepackage{enumitem}
\usepackage{bbm}
\usepackage{tikz-cd}
\usepackage{comment}

\def\today{\number\day\space\ifcase\month\or   January\or February\or
   March\or April\or May\or June\or   July\or August\or September\or
   October\or November\or December\fi\   \number\year}

\theoremstyle{definition}
\newtheorem{lma}{Lemma}[section]

\newaliascnt{thmCt}{lma}
\newtheorem{thm}[thmCt]{Theorem}
\aliascntresetthe{thmCt}

\newaliascnt{corCt}{lma}
\newtheorem{cor}[corCt]{Corollary}
\aliascntresetthe{corCt}

\newaliascnt{propCt}{lma}
\newtheorem{prop}[propCt]{Proposition}
\aliascntresetthe{propCt}

\newtheorem*{thm*}{Theorem}
\newtheorem*{cor*}{Corollary}
\newtheorem*{prop*}{Proposition}

\newcounter{theoremintro}
\newtheorem{thmintro}[theoremintro]{Theorem}

\newaliascnt{pgrCt}{lma}

\aliascntresetthe{pgrCt}

\newaliascnt{dfCt}{lma}
\newtheorem{df}[dfCt]{Definition}
\aliascntresetthe{dfCt}

\newaliascnt{remCt}{lma}
\newtheorem{rem}[remCt]{Remark}
\aliascntresetthe{remCt}

\newaliascnt{remsCt}{lma}

\aliascntresetthe{remsCt}

\newaliascnt{egCt}{lma}

\aliascntresetthe{egCt}

\newaliascnt{egsCt}{lma}

\aliascntresetthe{egsCt}

\newaliascnt{qstCt}{lma}

\aliascntresetthe{qstCt}

\newaliascnt{pbmCt}{lma}

\aliascntresetthe{pbmCt}

\newaliascnt{notaCt}{lma}
\newtheorem{nota}[notaCt]{Notation}
\aliascntresetthe{notaCt}

\newcommand{\beq}{\begin{equation}}
\newcommand{\eeq}{\end{equation}}
\newcommand{\beqa}{\begin{eqnarray*}}
\newcommand{\eeqa}{\end{eqnarray*}}
\newcommand{\bal}{\begin{align*}}
\newcommand{\eal}{\end{align*}}
\newcommand{\bi}{\begin{itemize}}
\newcommand{\ei}{\end{itemize}}
\newcommand{\be}{\begin{enumerate}}
\newcommand{\ee}{\end{enumerate}}

\newcommand{\Z}{{\mathbb{Z}}}
\newcommand{\R}{{\mathbb{R}}}
\newcommand{\C}{{\mathbb{C}}}
\newcommand{\N}{{\mathbb{N}}}

\newcommand{\B}{{\mathcal{B}}}

\newcommand{\crn}{{\mathsf{c}}}
\newcommand{\ran}{{\mathsf{r}}}
\newcommand{\sor}{{\mathsf{s}}}

\newcommand{\id}{{\mathrm{id}}}

\newcommand{\I}{\infty}

\title[]{Uniqueness theorems for $L^p$-operator graph algebras}

\date{\today}

\thanks{The first named author was partially supported by the Swedish Research Council Grant 2021-04561.}

\author[Eusebio Gardella]{Eusebio Gardella}
\address{Eusebio Gardella
Department of Mathematical Sciences, Chalmers University of
Technology and University of Gothenburg, Gothenburg SE-412 96, Sweden.}
\email{gardella@chalmers.se}
\urladdr{www.math.chalmers.se/~gardella}

\author{Siri Tinghammar}
\address{Siri Tinghammar,
Department of Mathematical Sciences, Chalmers University of
Technology and University of Gothenburg, Gothenburg SE-412 96, Sweden.}
\email{siriti@chalmers.se}
\begin{document}

\begin{abstract}
We continue the study of $L^p$-operator algebras associated with directed graphs initiated
by Corti\~nas and Rodr\'iguez, and we establish $L^p$-analogs of both the gauge-invariant and the Cuntz-Krieger uniqueness theorems. The first of these asserts that for a graph $Q$, a gauge-equivariant spatial representation of its Leavitt path algebra $L_Q$ on an $L^p$-space generates an injective representation whenever the idempotents associated to the vertices of $Q$ are nonzero. The second of these theorems states that, in the setting just described, the same conclusion holds if gauge-equivariance is replaced by the 
assumption that every cycle in $Q$ has an entry. Additionally, we show that for acyclic graphs, such representations are automatically isometric.

While our general approach is inspired by the proofs in the 
C*-algebra setting, a careful analysis of spatial representations of graphs on 
$L^p$-spaces is required. In particular, we exploit the interplay between analytical properties of Banach algebras, such as the role of hermitian elements, and geometric notions specific to 
$L^p$-spaces, such as spatial implementation.
\end{abstract}

\maketitle

\tableofcontents

\renewcommand*{\thetheoremintro}{\Alph{theoremintro}}
\section{Introduction}

Graph C*-algebras provide a natural framework for encoding the combinatorial structure of a directed graph into a C*-algebra. Given a directed graph $Q$, the associated graph C*-algebra $C^*(Q)$ is defined as the universal C*-algebra generated by partial isometries corresponding to edges and projections corresponding to vertices, satisfying relations that reflect the connectivity of the graph; see \cite{raeburn} for an introductory text. Graph algebras generalize both Cuntz-Krieger algebras and the Toeplitz-Cuntz algebra and appear naturally in various fields, including operator algebras, symbolic dynamics, and noncommutative geometry.
Graph C*-algebras have played an important role in the classification program for nuclear C*-algebras, particularly in the study of purely infinite algebras; see \cite{spielberg1, spielberg2, szymanski}. Their connections to groupoid algebras and dynamical systems have led to significant progress in understanding their structure and ideal lattice. 
A major breakthrough in the field was the completion, around a decade ago, of the classification of unital graph C*-algebras by filtered $K$-theory in \cite{eilers}, also providing a remarkable geometric classification in terms of moves on the graphs.
Current research focuses on studying certain intermediate equivalence relations (considering Cartan-preserving isomorphisms, or $\mathbb{T}$-equivariant ones, among others) and their connections to symbolic dynamics; see \cite{AER,matsumoto} and the recent survey \cite{brix}.

A fundamental question in the study of graph C*-algebras is the extent to which their algebraic structure (uniquely) determines the representations. One may naively expect that the relations defining $C^*(Q)$ are ``unique'', and that any family of operators satisfying the relations defining $C^*(Q)$ always generates an isomorphic copy of $C^*(Q)$. While this is true
for certain graphs (most prominently, the graphs giving rise to the Cuntz algebras $\mathcal{O}_n$), in general a representation of $C^*(Q)$ can fail to be faithful if it does not adequately capture the structure of paths in the graph. This issue is closely related to the ideal structure of $C^*(Q)$ and the possibility of modding out parts of the algebra corresponding to ``non-essential'' portions of the graph.

Two classical results play a central role in addressing this issue: the Cuntz–Krieger uniqueness theorem and the gauge-invariant uniqueness theorem. These provide necessary and sufficient conditions for a representation of $C^*(Q)$ to be faithful, under certain assumptions on the graph $Q$ and on the representation, respectively. More concretely, the former states that if every cycle in $Q$ has an entry, then any representation that does not annihilate any vertex projection is automatically injective. On the other hand, the latter gives a complementary criterion, ensuring faithfulness for representations that are compatible with the canonical gauge action.
Together, these theorems show how the combinatorial and dynamical properties of the graph are reflected in the (gauge-invariant) ideal structure of the associated C*-algebra.

A natural generalization of graph C*-algebras are the $L^p$-operator algebras associated to graphs, where Hilbert spaces as the natural ambient for Cuntz-Krieger families are replaced with $L^p$-spaces. These algebras were introduced by Corti\~nas-Rodriguez in \cite{cortinas rodriguez}, motivated by earlier work of Phillips on $L^p$-analogs of the Cuntz algebras \cite{phillips}, and accompanying a growing body of literature on the topic of algebras of operators on $L^p$-spaces which includes $L^p$-versions of group algebras \cite{gardella thiel,
GT_isomorphisms, GT_Lq}; AF-algebras \cite{phillips viola, GL_UHG}; (twisted) groupoid algebras \cite{choi gardella thiel, GL, BKM, HO}, and more. 
These $L^p$-versions exhibit several structural differences with their C*-analogs, for example lacking an involution, while still retaining aspects of their combinatorial or algebraic origins. 
Investigating $L^p$-versions of the Cuntz-Krieger uniqueness theorem is particularly interesting 
because faithfulness criteria for representations in this broader context remain largely unexplored. Moreover, 
establishing such results could provide new insights into the interplay between analytic properties of $L^p$-operator algebras arising from geometric or combinatorial data, and the algebraic or dynamical features of the underlying models. 

In this work, we prove $L^p$-versions of the classical gauge-invariant and Cuntz-Krieger uniqueness theorems for 
graph C*-algebras:

\begin{thmintro}(Gauge-invariant uniqueness).
    Let $p\in [1,\infty)$, let $Q$ be a graph, let $A$ be a unitizable $L^p$-operator algebra, and let $(S,T,E)$ be a Cuntz-Krieger $Q$-family in $A$ such that $E_v\neq 0$ for all $v\in Q^0$. Suppose that there is an isometric action $\beta\colon \mathbb{T}\rightarrow \text{Aut}(A)$ such that $\beta_z(S_a)=zS_a,\ \beta_z(T_a)=\overline{z}T_a$ and $\beta_z(E_v)=E_v$ for all $z\in\mathbb{T}$, for all $a\in Q^1$ and for all $v\in Q^0$. Then the canonical homomorphism $\pi\colon \mathcal{O}^p(Q)\rightarrow A$ is injective.
\end{thmintro}

When every cycle in $Q$ has an entry, injectivity of the canonical homomorphisms
follows automatically without assuming gauge-equivariance. 

\begin{thmintro} (Cuntz-Krieger uniqueness).
Let $p\in [1,\infty)$, let $Q$ be a graph in which each cycle has an entry, let $A$ be a unitizable $L^p$-operator algebra, and let $(S,T,E)$ be a Cuntz-Krieger $Q$-family in $A$ such that $E_v\neq 0$ for all $v\in Q^0$. Then the canonical homomorphism $\pi\colon \mathcal{O}^p(Q)\rightarrow A$ is injective.
\end{thmintro}

We largely follow the general arguments used for graph C*-algebras (for example, as 
presented in \cite{raeburn}), but the absence of an involution in $\mathcal{O}^p(Q)$ forces us 
to carry out a careful analysis of spatial representations of graphs on 
$L^p$-spaces. In doing this, the interplay between analytical properties of Banach algebras, such as the role of hermitian elements, and geometric notions specific to 
$L^p$-spaces, such as spatial implementation, plays a key role. 

An alternative proof of our Theorem~B can be obtained by adapting
the methods used for graph C*-algebras in \cite{kumjian pask raeburn},
using the recent results of Bardadyn, Kwa\'sniewski and McKee on 
Banach algebras associated to groupoids from \cite{BKM}, and specifically using injectivity
on the diagonal subalgebra (which, in our setting, amounts to looking at the C*-core of $\mathcal{O}^p(Q)$; see \cite{choi gardella thiel}). The proof that we present here is 
graph-theoretical, as it avoids the use of groupoids or C*-cores. Although the results in 
\cite{BKM} would allow us to prove Theorem~B in a more direct way, the machinery we develop
here is likely to be useful in other contexts, and it allows us to prove additional results, such as Theorems~C and~D below.

An important tool in our work is the existence of a distinguished circle action $\gamma\colon\mathbb{T}\to\mathrm{Aut}(\mathcal{O}^p(Q))$, called the gauge action.
Showing injectivity of the restriction of the canonical 
homomorphisms to the fixed point 
algebra of this action is one of the main steps in showing injectivity, and for this 
the following result is very helpful.

\begin{thmintro}
Let $p\in [1,\infty)$ and let $Q$ be a graph without infinite emitters. 
Then the fixed point algebra
$\mathcal{O}^p(Q)^\gamma$ of the gauge action is a spatial $L^p$-AF-algebra. 
\end{thmintro}

We cannot conclude in general that the representation in Theorem~B is isometric, since for homomorphisms between $L^p$-operator algebras, injectivity does not in general imply isometry. On
the other hand, for graphs without cycles, this conclusion does indeed hold:

\begin{thmintro}
    Let $p\in [1,\infty)$, and let $Q$ be an acyclic graph. Let $(S,T,E)$ be a Cuntz-Krieger $Q$-family in a unitizable $L^p$-operator algebra $A$ such that $E_v\neq 0$ for all $v\in Q^0$. Then the canonical homomorphism $\pi\colon\mathcal{O}^p(Q)\rightarrow A$ is an isometry.
\end{thmintro}

We work with not-necessarily countable graphs throughout, and thus we spend some time extending the basic machinery developed in \cite{cortinas rodriguez} for the separable setting. Most arguments can be obtained with minor modifications, for example by replacing $\sigma$-finite measures with
localizable ones and using the version of Lamperti's theorem for such measures from \cite{choi gardella thiel}. There is only one place where
countability of a graph is crucial, namely the disingularization argument in
\autoref{desingularize}. To circumvent this, we will show that the $L^p$-graph algebra of an arbitrary graph can be
written as a direct limit, with isometric homomorphisms, of the $L^p$-graph algebras of certain countable Cuntz-Krieger subgraphs; see \autoref{cor:DirLim}.

\vspace{.2cm}

\textbf{Acknowledgements:} We would like to thank Jan Gundelach for a number of helpful discussions, as well as Krzysztof Bardadyn and Bartosz Kwa\'sniewski for making us aware of the connections of our work with theirs \cite{BKM}, and indicating to us an alternative way to prove Theorem~B. 

\section{Preliminaries on partial isometries and idempotents}

In this section, we collect a number of definitions and results that will be used in later sections.
We begin by setting some terminology around $L^p$-operator algebras, and by introducing certain elements which will be of particular interest, namely spatial partial isometries and hermitian idempotents. As we will see in the end of this section, spatial partial isometries and hermitian idempotents are closely connected.

\begin{df} (\cite[Definition 2.2]{phillips viola})
    Let $A$ be a Banach algebra and let $p\in[1,\infty)$. We say that $A$ is an \emph{$L^p$-operator algebra} if there exist an $L^p$-space $\mathcal{E}$ and an isometric homomorphism $A\rightarrow\mathcal{B}(\mathcal{E})$ of complex algebras. A \emph{representation} of $A$ on an $L^p$-space $\mathcal{E}$ is a contractive homomorphism $\varphi\colon A\rightarrow\mathcal{B}(\mathcal{E})$. The representation $\varphi$ is said to be \emph{nondegenerate} if $\text{span}\{\varphi(a)\xi\colon a\in A,\ \xi \in \mathcal{E}\}$ is dense in $\mathcal{E}$.
\end{df}

When $A$ is unital, a representation is nondegenerate if and only if it is unital, and one can always create a unital representation from a nonunital one, as shown in the next lemma. 
Given a Banach space $X$ and an idempotent $q\in \mathcal{B}(X)$, 
we write $\crn_q\colon \mathcal{B}(X)\rightarrow\mathcal{B}(qX)$ for the compression map, 
which is given by $\crn_q(a)=qaq$ for all $a\in\mathcal{B}(X)$.

\begin{lma}\label{unitalrep}
    Let $p\in [1,\infty)$, let $A$ be a unital $L^p$-operator algebra and let $\varphi\colon A\to \mathcal{B}(\mathcal{E})$ be a representation on an $L^p$-space $\mathcal{E}$. Set $q=\varphi(1)$. Then $q\mathcal{E}$ is an $L^p$-space and $\crn_q\circ\varphi$ is a unital representation of $A$ satisfying $\|\crn_q(\varphi(a))\|=\|\varphi(a)\|$ for all $a\in A$.
\end{lma}

\begin{proof}
Note that $\|q\|\leq 1$ by contractivity of $\varphi$. The fact that $q\mathcal{E}$ is an $L^p$-space follows from \cite[Theorem 6]{tzafriri}. It is clear that $\crn_q\circ\varphi$ is a unital representation of $A$.
Since $\varphi(a)=q\varphi(a)q$ for all $a\in A$, we deduce that $\crn_q$ is an isometry on $\text{Im}(\varphi)$, as desired.
\end{proof}

\begin{df}
    Let $A$ be a ring and let $e,f\in A$ be idempotents. We say that $e$ and $f$ are \emph{orthogonal} if $ef=fe=0$.
\end{df}

\begin{df}\label{partialisometry}
    Let $A$ be an algebra. We say that $s\in A$ is a \emph{partial isometry} if there exists an element $t\in A$, called a \emph{reverse} of $s$, such that $sts=s$ and $tst=t$. 
    
    If $s_1$ and $s_2$ are partial isometries, we say that they have \emph{orthogonal ranges} if there exist reverses $t_1$ and $t_2$ with $s_it_j=t_js_i=0$ for $i,j=1,2$ and $i\neq j$.
\end{df}

Note that if $s$ is a partial isometry with reverse $t$, then $st$ and $ts$ are idempotents.

Let $(X,\mathcal{B},\mu)$ be a localizable measure space (see, for example, \cite[Definition~2.2]{GarThi_isomorphisms} for the precise definition; see also \cite[Remark~2.5]{choi gardella thiel}). For $E\in\mathcal{B}$, we write $\mathcal{B}_E=\{E\cap F\colon F\in \mathcal{B}\}$, and we let $\mu_E$ denote the restriction of $\mu$ to $\mathcal{B}_E$. Then $(E,\mathcal{B}_E,\mu_E)$ is again a localizable measure space. Note that the map $L^p(\mu)\rightarrow L^p(\mu_E)\oplus L^p(\mu_{E^c})$ given by $f\mapsto (f|_E, f|_{E^c})$ is an isometric isomorphism.

Let $(X,\mathcal{B},\mu)$ and $(Y,\mathcal{C},\nu)$ be localizable measure spaces. A \emph{measurable set transformation} from $X$ to $Y$ is a homomorphism $\eta\colon\mathcal{B}\rightarrow\mathcal{C}$ of $\sigma$-algebras. If $\eta$ is bijective, then $\mu\circ\eta^{-1}$ is a localizable measure on $Y$, and we have $\mu\circ\eta^{-1}\ll\nu$. Hence the Radon-Nikodym theorem for localizable measures (\cite[Theorem 2.7]{GarThi_isomorphisms}) applies in the definition below. 

Write $\mathcal{B}(\mathbb{R})$ for the $\sigma$-algebra of Borel subsets of $\mathbb{R}$. If $f\colon (X,\mathcal{B})\rightarrow (\R,\mathcal{B}(\mathbb{R}))$ is a measurable function, we may identify $f$ with a $\sigma$-algebra homomorphism $\tilde{f}\colon\mathcal{B}(\R)\rightarrow \mathcal{B}$ (see \cite[Example 2.6]{gardella}). We write $f\circ\eta$ for the composition $\eta\circ\tilde{f}$, which may then be viewed as a measurable function $Y\rightarrow\R$, again
using \cite[Example 2.6]{gardella}.

\begin{df}(\cite[Definition 6.2]{gardella})\label{spatialpartialiso}
    Let $p\in[1,\infty)$ and let $(X,\mathcal{B},\mu)$ be a localizable measure space. Let $E,F\in\mathcal{B}$, let $\eta\colon \mathcal{B}_E\rightarrow\mathcal{B}_F$ be a bijective measureble set transformation, and let $f\in\mathcal{U}(L^\infty(\mu_F))$. We call the tuple $(E,F,\eta,f)$ a \emph{spatial system}. The \emph{spatial partial isometry} associated to the spatial system $(E,F,\eta,f)$ is the contractive map $s\colon L^p(\mu)\rightarrow L^p(\mu)$ which vanishes on $L^p(\mu_{E^c})$, and which is given on $L^p(\mu_E)$ by the isometric isomorphism $L^p(\mu_E)\rightarrow L^p(\mu_F)$ defined by
    $$\sor(\xi)(y)=f(y)(\xi\circ\eta)(y)\left(\frac{d(\mu_E\circ\eta^{-1})}{d\mu_F}(y)\right)^{\frac{1}{p}}$$
    for all $\xi\in L^p(\mu_E)$ and for all $y\in F$. A \emph{spatial idempotent} is an idempotent which is also a spatial partial isometry.
\end{df}

\begin{rem}\label{spatialreverse}
    Let $s$ be the spatial partial isometry associated to the spatial system $(E,F,\eta,f)$. Then $s$ is a partial isometry in the sense of \autoref{partialisometry}. It has a unique reverse $t$, which is the spatial partial isometry associated to the tuple $(F,E,\eta^{-1},\overline{f}\circ\eta^{-1})$. Letting $m_{\mathbbm{1}_X}$ denote the multiplication operator associated to the characteristic function $\mathbbm{1}_X$ of $X\in\mathcal{B}$, we have $st=m_{\mathbbm{1}_F}$ and $ts=m_{\mathbbm{1}_E}$. If $p=2$, the reverse of a spatial partial isometry $s$ is simply its adjoint $t=s^*$.
\end{rem}

\begin{rem}\label{spatialsemigroup}
    The product of spatial partial isometries is again a spatial partial isometry. Indeed, suppose $s_1$ and $s_2$ are spatial partial isometries associated to the spatial systems $(E_1,F_1,\eta_1,f_1)$ and $(E_2,F_2,\eta_2,f_2)$, respectively. Then $s_2s_1$ is the spatial partial isometry associated to $(\eta_1^{-1}(E_2\cap F_1),\eta_2(E_2\cap F_1),\eta_2\circ\eta_1,f_2(f_1\circ\eta_2))$. In particular, $s_2s_1=0$ if and only if $E_2\cap F_1=\emptyset$.
\end{rem}

The following is a consequence of Lamperti's theorem for localizable measures; see 
\cite[Theorem 3.7]{GarThi_isomorphisms}.

\begin{prop}(\label{spatial}\cite[Proposition 6.5]{gardella})
    Let $p\in [1,\infty)\setminus \{2\}$, let $\mathcal{E}$ be a localizable $L^p$-space and let $s\in\mathcal{B}(\mathcal{E})$. Then $s$ is a spatial partial isometry if and only if $s$ is contractive and there exists $t\in\mathcal{B}(\mathcal{E})$ contractive with $sts=s$ and $tst=t$, and such that $st$ and $ts$ are spatial idempotents.
\end{prop}

The following lemma will be needed later on.

\begin{lma}\label{normorthspatial}
    Let $p\in[1,\infty)$, let $(X,\mathcal{B},\mu)$ be a localizable measure space, let $s_1,\dots,s_n\in\mathcal{B}(L^p(\mu))$ be spatial partial isometries with orthogonal ranges, and let $\lambda\in\C^n$. Then $\|\sum_{i=1}^n\lambda_i s_i\|=\|\lambda\|_\infty$.
\end{lma}

\begin{proof}
    For $i=1,\dots,n$, let $(E_i,F_i,\eta_i,f_i)$ be the spatial system associated to $s_i$. Let $t_i$ be the reverse of $s_i$, associated to the system $(F_i,E_i,\eta_i^{-1},\overline{f_i}\circ\eta_i^{-1})$. If $t_i s_j=0$, by \autoref{spatialsemigroup} we must have $E_i\cap E_j=\emptyset$. As $t_i s_j=0$ for all $i\neq j$, the supports of the $s_i$ are pairwise disjoint. Hence $||\sum_{i=1}^n\lambda_i s_i||=\max_{1\leq i\leq n}\|\lambda_i s_i\|=\|\lambda\|_\infty$.
\end{proof}

\begin{df}\label{hermitiandef}
    Let $A$ be a unital Banach algebra. An element $a\in A$ is called \emph{hermitian} if $\|\text{exp}(i\lambda a)\|\leq 1$ for all $\lambda\in\R$. An element which is both hermitian and an idempotent is called a \emph{hermitian idempotent}.
\end{df}

Note that an element in a $C^*$-algebra is hermitian if and only if it is self-adjoint.

\begin{rem}\label{preservehermitian}
    Let $A$ and $B$ be unital Banach algebras, and let $\varphi\colon A\rightarrow B$ be a unital contractive homomorphism. Suppose $a\in A$ is hermitian. Then for $\lambda\in\R$, we have 
    $$\|\text{exp}(i\lambda\varphi(a))\|=\|\varphi(\text{exp}(i\lambda a))\|\leq \|\text{exp}(i\lambda a)\|\leq 1,$$ 
    so that $\varphi(a)$ is hermitian.
    Moreover, if $a,b\in A$ are hermitian and $\lambda\in\R$, then 
    $$\|\text{exp}(i\lambda(a\pm b))\|\leq \|\text{exp}(i\lambda a)\|\|\text{exp}(i(\pm\lambda) b)\|\leq 1.$$
    Thus, sums and differences of hermitian elements are again hermitian.
\end{rem}

As stated in the proposition below, for $p\in[1,\infty)\setminus\{2\}$ the hermitian idempotents in $\mathcal{B}(L^p(\mu))$ are precisely the spatial idempotents. For $p=2$, the hermitian idempotents are precisely the orthogonal projections, which is a larger class than the spatial idempotents.

\begin{prop}\label{hermitianequiv}
    Let $p\in[1,\infty)\setminus\{2\}$, let $(X,\mathcal{B},\mu)$ be a localizable measure space, and let $e\in \mathcal{B}(L^p(\mu))$ be an idempotent. Then the following are equivalent:
    \begin{enumerate}
        \item $e$ is a hermitian idempotent,
        \item $e$ is a spatial idempotent,
        \item there is $E\in \mathcal{B}$ such that $e=m_{\mathbbm{1}_E}$.
    \end{enumerate}
\end{prop}
\begin{proof} The proof is identical to that of \cite[Lemma 6.9]{phillips viola}, which gives
 the same statement for $\sigma$-finite measure spaces, applying the version of Lamperti's theorem
 for localizable measures (\cite[Theorem~3.7]{GarThi_isomorphisms}) instead of the version for 
 $\sigma$-finite spaces. We omit the details.
\end{proof}

We have so far only defined hermitian elements in \emph{unital} Banach algebras, since 
\autoref{hermitiandef} uses exponentials via 
holomorphic functional calculus, and therefore requires the existence of a unit. On the other hand, in the following sections we will need a notion of hermitian elements in a not necessarily unital $L^p$-operator algebra. 
For this, we will consider a unitization of the algebra in question, but one must be careful since norms on unitizations are not canonical, and whether an element is hermitian depends on the norm used. For an illustration of this, see \cite[Example 6.12]{phillips viola}. 

These considerations lead us to the following definition. We will write $A^+$ for the minimal (one-dimensional) unitization of a Banach algebra $A$, and we will write $M(A)$ for its 
multiplier algebra (with the norm induced by regarding multipliers on $A$ as bounded linear
maps $A\to A$).

\begin{df}\label{df:unitizable} 
Let $p\in [1,\infty)$ and let $A$ be an $L^p$-operator algebra. We say that $A$ is \emph{unitizable} (as an $L^p$-operator algebra), if $A^+$ is an $L^p$-operator algebra when
endowed with the norm induced by the inclusion $A^+\subseteq M(A)$.
\end{df}

We highlight a large class of unitizable $L^p$-operator algebras:

\begin{rem}
For any $p\in [1,\I)$, every $L^p$-operator algebra $A$ with a contractive approximate identity which
admits a nondegenerate representation on an $L^p$-space is 
unitizable; in fact, in this case $M(A)$ is 
an $L^p$-operator algebra by \cite[Theorem 4.1]{gardella thiel repr}. The assumption that there
is a nondegenerate representation on an $L^p$-space is automatic of $p>1$ by \cite[Theorem 4.3]{gardella thiel repr}, and it is also automatic if $A$ contains an approximate identity consisting of idempotents as in \cite[Definition~9.7]{phillips viola}, by \cite[Definition~9.9]{phillips viola}.
\end{rem}

We will show in \autoref{stespatial} that $L^p$-operator graph algebras are unitizable, by showing that they admit nondegenerate
representations on $L^p$-spaces.
Although we will not explicitly need it in our work, we 
also point out that $L^p$-operator graph algebras also 
contain an approximate identity consisting of idempotents as in \cite[Definition~9.7]{phillips viola}\footnote{Such an approximate identity can be constructed by taking finite sums of the idempotents associated to the vertices of the graph.}.

We now define hermitian elements in unitizable $L^p$-operator algebras in the expected way. For a Banach algebra $A$, we always endow $A^+$ with the norm
induced by its inclusion into $M(A)$.

\begin{df}\label{hermitiannonunital}
Let $p\in[1,\infty)$ and let $A$ be a unitizable $L^p$-operator
algebra. We say that an element $a\in A$ is \emph{hermitian} in $A$ if it is hermitian in $A^+$.
\end{df}

\section{$L^p$-operator graph algebras and Cuntz-Krieger families}

In this section, we review the definition of $L^p$-operator graph
algebras and develop some theory in order to prove their universal property in terms of what we will call spatial Cuntz-Krieger families.

\begin{df}
A \emph{directed graph} $Q$ is a tuple $Q=(Q^0,Q^1,\ran,\sor)$ consisting of a set $Q^0$ of \emph{vertices}, a set $Q^1$ of \emph{edges}, and functions $\sor,\ran\colon Q^1\rightarrow Q^0$ called the \emph{source} and \emph{range} functions, respectively. 

We say that $Q$ is \emph{finite} or \emph{countable} if both $Q^0$ and $Q^1$ are finite or countable, respectively. A vertex $v\in Q^0$ is called:
\begin{enumerate}
    \item a \emph{source} if $\ran^{-1}(v)=\emptyset$,
    \item a \emph{sink} if $\sor^{-1}(v)=\emptyset$,
    \item an \emph{infinite emitter} if $\sor^{-1}(v)$ is infinite,
    \item an \emph{infinite receiver} if $\ran^{-1}(v)$ is infinite,
    \item \emph{regular} if $v$ is not a source or an infinite receiver.
\end{enumerate}
We write $\text{Reg}(Q)$ for the set of regular vertices of $Q$, and we say that $Q$ is regular if $\text{Reg}(Q)=Q^0$. 
\end{df}

Since we will only work with directed graphs, we shall use the term graph to refer to a directed graph. Also, if the graph $Q$ is not clear from the context (which will only happen in the proof of 
\autoref{finitegraphembedding} and \autoref{prop:DirLim}), we write $\ran_Q$ and $\sor_Q$ for the functions $\ran$ and $\sor$, respectively. 

We now proceed to define paths, and in particular establish our
convention for concatenation of edges. For this, we will regard
edges as morphisms and concatenate them from right to left, 
which agrees with the convention used in \cite{raeburn} and is opposite of that in \cite{cortinas rodriguez}.\footnote{One consequence of the difference in conventions is 
what is here called an ``entry'' is
actually an ``exit'' in \cite{cortinas rodriguez}.}

\begin{df}
Let $Q$ be a graph. A \emph{path} in $Q$ is a (finite or infinite) sequence $\alpha =(\alpha_n)_{n\geq 1}$ of edges $\alpha_n\in Q^1$, such that $\sor(\alpha_n)=\ran(\alpha_{n+1})$ for all $n\geq 1$. We identify a path $\alpha=(\alpha_n)_{n\geq 1}$ with the formal expression $\alpha_1\alpha_2\cdots\alpha_n\cdots$. We set $\ran(\alpha)=\ran(\alpha_1)$, and if $\alpha=\alpha_1\cdots\alpha_n$ is finite, we write $|\alpha|=n$ and $\sor(\alpha)=\sor(\alpha_n)$. 

A finite path $\alpha=\alpha_1\cdots\alpha_n$ is called a \emph{cycle} if $\ran(\alpha)=\sor(\alpha)$ and $\alpha_i\neq \alpha_j$ for $i\neq j$. An edge $a\in Q^1$ is called an \emph{entry} for the cycle $\alpha$ if there exists $1\leq i\leq n$ such that $\ran(a)=\ran(\alpha_i)$ but $a\neq\alpha_i$. If $Q$ contains no cycles, then it is called \emph{acyclic}. 
\end{df}

\begin{nota} \label{nota:paths}
We denote by $Q^*$ the set of finite paths of $Q$, and we denote by $Q^n$ the set of paths of length $n$, considering vertices as paths of length $0$. Hence $Q^*=\bigcup_{n\geq 0}Q^n$. The set $ Q^*$ is equipped with the preorder given by $\alpha\leq \beta$ whenever there exists $\gamma\in Q^*$ such that $\alpha=\beta\gamma$.
\end{nota}

Next, we define the Leavitt path algebra associated to a graph.

\begin{df}\label{leavittpathalg}
    Let $Q$ be a graph. Its \emph{Leavitt path algebra} $L_Q$ is the universal complex algebra generated by elements $s_a,t_a$ and $e_v$ for $a\in Q^1$ and for $v\in Q^0$, subject to the following relations:
    \begin{enumerate}
        \item $e_ve_w=\delta_{v,w}e_v$ for all $v,w\in Q^0$,
        \item $e_{r(a)}s_a=s_ae_{s(a)}=s_a$ and $t_ae_{r(a)}=e_{s(a)}t_a=e_a$ for all $a\in Q^1$,
        \item $t_as_b=\delta_{a,b}e_{\sor(b)}$ for all $a,b\in Q^1$ (CK1),
        \item $e_v=\sum_{a\in \ran^{-1}(v)}s_at_a$ for all regular vertices $v\in \mathrm{Reg}(Q)$ (CK2).
    \end{enumerate}
\end{df}

\begin{rem}\label{tsrelationcomparable}
Note that the relations in \autoref{leavittpathalg} imply that for $\alpha,\beta\in Q^*$ we have
\begin{equation*}
    t_\alpha s_\beta=
    \begin{cases}
        e_{\sor(\alpha)} &\text{if }\alpha=\beta,\\
        t_{\alpha'} &\text{if }\alpha=\beta\alpha',\\
        s_{\beta'} &\text{if }\beta=\alpha\beta'\\
        0 &\text{else}.
    \end{cases}    
\end{equation*}
Hence $t_\alpha s_\beta\neq 0$ if and only if $\alpha$ and $\beta$ are comparable under the preorder on $Q^*$ described in \autoref{nota:paths}.
\end{rem}

\begin{df}\label{spatialckfam}
    Let $p\in[1,\infty)$, let $Q$ be a graph and let $A$ be a unitizable $L^p$-operator algebra. A \emph{Cuntz-Krieger $Q$-family} in $A$ is a family 
    \[(S,T,E)=\big(\{S_a\}_{a\in Q^1},\{T_a\}_{a\in Q^1},\{E_v\}_{v\in Q^0}\big)\]
    of elements in $A$ such that
\begin{enumerate}
 \item[(i)] $\{E_v\colon v\in Q^0\}$ is a family of pairwise orthogonal hermitian idempotents;
 \item[(ii)] for each $a\in Q^1$, we have $\|S_a\|\leq 1$ and $\|T_a\|\leq 1$;
 \item[(iii)] for each $a\in Q^1$, the operator $S_a$ is a partial isometry with reverse $T_a$, and both $S_aT_a$ and $T_aS_a$ are hermitian idempotents,
\end{enumerate}
satisfying the following relations:
    \begin{enumerate}
        \item $E_{\ran(a)}S_a=S_aE_{\sor(a)}=S_a$ and $T_aE_{\ran(a)}=E_{\sor(a)}T_a=T_a$ for all $a\in Q^1$,
        \item $T_aS_b=\delta_{a,b}E_{\sor(b)}$ for all $a,b\in Q^1$ (CK1),
        \item $E_v=\sum_{a\in \ran^{-1}(v)}S_aT_a$ for all regular vertices $v\in \mathrm{Reg}(Q)$ (CK2).
    \end{enumerate}
\end{df}

We now identify the relevant classes of representations of graphs on $L^p$-operator algebras.

\begin{df}\label{lpspatialrep}
    Let $p\in [1,\infty)$, let $Q$ be a graph and let $A$ be a unitizable $L^p$-operator algebra. A \emph{representation} of $L_Q$ on $A$ is an algebra homomorphism $L_Q\rightarrow A$ such that $(\varphi(s),\varphi(t),\varphi(e))$ is a Cuntz-Krieger $Q$-family in $A$. 
    
    If $(X,\mathcal{B},\mu)$ is a localizable measure space and $A=\mathcal{B}(L^p(\mu))$, then we say that $\varphi$ is \emph{spatial} if $(\varphi(s),\varphi(t),\varphi(e))$ is a family of spatial partial isometries in $A$.
\end{df}

\begin{rem}\label{differentreps}
    If $p\in[1,\infty)\setminus\{2\}$, then any representation of $L_Q$ on an $L^p$-space is spatial by \autoref{hermitianequiv} and \autoref{spatial}. If $p=2$, then a representation is precisely a homomorphism such that the $\varphi(e_v)$ are mutually orthogonal projections, the $\varphi(s_a)$ are partial isometries and $\varphi(t_a)=\varphi(s_a)^*$ for each $a\in Q^1$. Thus, \autoref{lpspatialrep} extends the usual notion of Cuntz-Krieger families from C*-algebras to $L^p$-operator algebras.
\end{rem}

In the following proposition, for $p\neq 2$ we obtain a spatial
description of Cuntz-Krieger families.

\begin{prop}\label{orthogonal}
Let $p\in[1,\infty)\setminus\{2\}$, let $Q$ be a graph, let $(X,\mathcal{B},\mu)$ be a localizable measure space and let $\varphi\colon L_Q\rightarrow\mathcal{B}(L^p(\mu))$ be a representation, and let $\alpha,\beta\in Q^*$. 
Then:
\be\item The element $\varphi(s_\alpha t_\alpha)\in \mathcal{B}(L^p(\mu))$ is a spatial idempotent, so there exists $X_\alpha\in\mathcal{B}$ with $\varphi(s_\alpha t_\alpha)=m_{\mathbbm{1}_{X_\alpha}}$.
\item If $\alpha\leq \beta$, then $X_\alpha\subseteq X_\beta$.
\item If $\alpha$ and $\beta$ are incomparable, then $X_\alpha\cap X_\beta=\emptyset$. Moreover, $\varphi(s_\alpha)$ and $\varphi(s_\beta)$ are partial isometries with orthogonal ranges.
\ee\end{prop}
\begin{proof}    
Note that $\varphi(s_\alpha t_\beta)$ is a spatial partial isometry for any $\alpha,\beta\in Q^*$, as the set of spatial partial isometries is a semigroup by \autoref{spatialsemigroup}.

(1). This follows directly from the line above together with \autoref{hermitianequiv}.

(2). One can readily that $\varphi(s_\beta)$ is the spatial partial isometry associated to a spatial system of the form $(X_{\sor(\beta)},X_\beta,\eta_\beta,f_\beta)$, where $\eta_\beta$ is a measurable set transformation from $X_{\sor(\beta)}$ to $X_\beta$, and where $f_\beta\in\mathcal{U}(L^\infty(X_\beta))$. 

Suppose that $\alpha\leq\beta$, and choose $\gamma\in Q^*$ with $\alpha=\beta\gamma$. Then $X_{\sor(\beta)}=\text{supp}(\eta_\beta),$ and therefore 
$$X_\alpha=\eta_\beta(X_\gamma)\subseteq \eta_\beta(X_{\sor(\beta)})=X_\beta.$$ 

(3).
If $\alpha$ and $\beta$ are incomparable, by \autoref{tsrelationcomparable} we have 
$$0=\varphi(s_\alpha t_\alpha s_\beta t_\beta)=m_{\mathbbm{1}_{X_\alpha}}m_{\mathbbm{1}_{X_\beta}},$$ 
so $X_\alpha$ and $X_\beta$ are disjoint. By \autoref{spatialreverse} and \autoref{spatialsemigroup}, this implies that $\varphi(s_\alpha)$ and $\varphi(s_\beta)$ are partial isometries with orthogonal ranges. 
\end{proof}

\begin{rem}\label{propertiesspatialfam}
    Let $p\in[1,\infty)\setminus\{2\}$, let $Q$ be a graph and let $(S,T,E)$ be a Cuntz-Krieger $Q$-family in a unitizable $L^p$-operator algebra $A$. By \autoref{unitalrep} there is a unital isometric representation $\varphi$ of $A^+$ on an $L^p$-space $\mathcal{E}$. By \cite[Proposition~2.6]{choi gardella thiel}, we can realize $\mathcal{E}$ as the $L^p$-space of a localizable measure.
    By \autoref{preservehermitian} the map $\varphi$ preserves hermitian elements, so by \autoref{differentreps} the family $(\varphi(S),\varphi(T),\varphi(E))$ consists of spatial partial isometries in $\mathcal{B}(\mathcal{E})$. It follows from parts~(1) and~(3) of~\autoref{orthogonal} that for each $\alpha\in Q^*$ the element $S_\alpha T_\alpha$ is a hermitian idempotent, and if $\alpha_1,\dots,\alpha_n$ are incomparable paths then $S_{\alpha_1},\dots,S_{\alpha_n}$ are partial isometries with orthogonal ranges. By \autoref{normorthspatial}, for $\lambda\in\C^n$ we have $\|\sum_{i=1}^n \lambda_i S_{\alpha_i}\|=\|\lambda\|_\infty$.
\end{rem}

The following definition is similar to the one presented in \cite[Section 7]{cortinas rodriguez}, with a slight modification which allows us to include the case $p=2$ in our discussion.

\begin{df}\label{opq}
    Let $p\in[1,\infty)$, and let $Q$ be a graph. The \emph{$L^p$-operator graph algebra} $\mathcal{O}^p(Q)$ of $Q$ is the completion of $L_Q$ in the norm 
    $$\|x\|=\sup \big\{\|\varphi(x)\|\colon \varphi\text{ is a representation of } L_Q \text{ on an }L^p\text{-space} \big\}.$$
\end{df}

By \autoref{differentreps}, the $L^2$-operator graph algebra $\mathcal{O}^2(Q)$ is precisely the graph $C^*$-algebra $C^*(Q)$ as defined in \cite[Proposition 1.21]{raeburn}. Since the results in this paper are already known for $p=2$ (as can be found in \cite{raeburn}), we will largely restrict ourselves to the case $p\in[1,\infty)\setminus\{2\}$ in our proofs. For these values of $p$, we can take advantage of the geometric properties of spatial representations.

The following is straightforward, so we only sketch an idea of the proof.

\begin{lma}\label{opqnormproperties}
    Let $p\in[1,\infty)$, and let $Q$ be a graph. Then $\|\cdot\|$ is a norm, and $\mathcal{O}^p(Q)$ admits an isometric nondegenerate representation on the $L^p$-space of a localizable measure. In particular, $\mathcal{O}^p(Q)$ is an $L^p$-operator algebra.  
\end{lma}

\begin{proof}
Since the result is well-known for $p=2$, we may assume that 
$p\in [1,\infty)\setminus\{2\}$.
Let $X$ be any set with cardinality at least equal to the cardinality of 
$Q$. Then the construction in the proof of \cite[Proposition 4.23]{cortinas rodriguez} can be generalized to uncountable graphs (replacing $\mathbb{N}$ with $X$), along the lines of \cite[Lemma 1.5]{goodearl}, to yield an injective representation $L_Q\rightarrow\mathcal{B}(\ell^p(X))$. It follows that $\|\cdot\|$ is indeed a norm. 

Arguing exactly as in \cite[Lemma 4.12]{cortinas rodriguez}, one can show that each spatial representation of $L_Q$ factors as the composition of a nondegenerate spatial representation and an isometry, so in \autoref{opq} it suffices to take the supremum over nondegenerate representations. Let $D$ be a dense subset of $L_Q$ with respect to the norm on $\mathcal{O}^p(Q)$. For each $d\in D$ and for each $n\in\N$, we may find a localizable measure space $(X_{d,n},\mathcal{B}_{d,n},\mu_{d,n})$ and a nondegenerate representation $\varphi_{d,n}\colon L_Q\rightarrow\mathcal{B}(L^p(\mu_{d,n}))$ of $L_Q$ such that $\Big |\|\varphi_{d,n}(d)\|-\|d\|\Big |<\frac{1}{n}$. Then 
    $$\varphi=\bigoplus_{\substack{d\in D}}\bigoplus_{n\in \N}\varphi_{d,n}\colon L_Q\rightarrow\mathcal{B}\Big(L^p\Big(\bigsqcup_{\substack{d\in D}}\bigsqcup_{n\in \N} \mu_{d,n}\Big)\Big)$$ 
    is a nondegenerate isometric representation of $L_Q$, which may be extended to a nondegenerate isometric representation of $\mathcal{O}^p(Q)$. Note that $\bigsqcup_{n\in \N} \mu_{d,n}$ is localizable.
\end{proof}

\begin{lma}\label{opqapproxid}
    Let $p\in[1,\infty)$ and let $Q$ be a graph. Let $\mathcal{S}$ be the collection of finite subsets of $Q^0$, and for $S\in \mathcal{S}$ let $\varepsilon_S=\sum_{v\in S}e_v$. Then $(\varepsilon_S)_{S\in \mathcal{S}}$ is a contractive approximate identity for $\mathcal{O}^p(Q)$.
\end{lma}

\begin{proof}
Since the result is known for $p=2$, we may assume that 
$p\in [1,\infty)\setminus\{2\}$.
Let $F$ be the set of representations of $L_Q$ on $L^p$-spaces. For each $\varphi\in F$, the element $\varphi(\varepsilon_S)$ is a sum of pairwise orthogonal spatial idempotents. Hence $\|\varepsilon_S\|=\sup_{\varphi\in F}\|\varphi(\varepsilon_S)\|=1$. Let $a\in \mathcal{O}^p(Q)$ and $\delta>0$ be given. Choose $b\in L_Q$ with $\|a-b\|<\frac{\delta}{2}$, and choose $R\in\mathcal{S}$ such that $\varepsilon_Sb=b=b\varepsilon_S$ for $S\supseteq R$. Then we have $\|\varepsilon_Sa-a\|\leq \|\varepsilon_S a-\varepsilon_S b\|+\|b-a\|<\delta$ for $S\supseteq R$, and similarly $\|a\varepsilon_S-a\|<\delta$.
\end{proof}

\begin{prop}\label{stespatial}
    Let $p\in[1,\infty)$ and let $Q$ be a graph. Then 
    $\mathcal{O}^p(Q)$ is unitizable. Moreover, the generating family $(s,t,e)$ of $\mathcal{O}^p(Q)$ is a Cuntz-Krieger $Q$-family.
\end{prop}

\begin{proof}
As before, we may assume that $p\in [1,\I)\setminus\{2\}$.
By \autoref{opqapproxid}, $\mathcal{O}^p(Q)$ has a contractive approximate identity. By \autoref{opqnormproperties}, there exist an $L^p$-space $\mathcal{E}$ and a nondegenerate isometric representation of $L_Q$ on $\mathcal{E}$, which extends to an isometric representation $\varphi$ of $\mathcal{O}^p(Q)$. By \cite[Theorem~4.1]{gardella thiel repr}, the induced representation $\varphi^+\colon \mathcal{O}^p(Q)^+\rightarrow\mathcal{B}(\mathcal{E})$ is isometric, and hence $\mathcal{O}^p(Q)$ is unitizable.

For each $a\in Q^1$ we have $\|s_a\|=\|\varphi(s_a)\|=1$ and similarly $\|t_a\|=1$. As $\varphi^+$ is isometric and unital, it preserves hermitian elements, so $e_v$ and $s_at_a$ are hermitian idempotents in $\mathcal{O}^p(Q)$ for each $v\in Q^0$ and for each $a\in Q^1$. 
\end{proof}

\begin{prop}\label{opquniversal}
    Let $p\in[1,\infty)$ and let $Q$ be a graph. The graph $L^p$-operator algebra $\mathcal{O}^p(Q)$ generated by the Cuntz-Krieger $Q$-family $(s,t,e)$ has the following universal property: for any Cuntz-Krieger $Q$-family $(S,T,E)$ in an $L^p$-operator algebra $A$, there is a (unique) contractive homomorphism $\pi\colon \mathcal{O}^p(Q)\rightarrow A$ satisfying $\pi(s_a)=S_a,\ \pi(t_a)=T_a$ and $\pi(e_v)=E_v$ for all $a\in Q^1$ and for all $v\in Q^0$.
\end{prop}

\begin{proof}
    Choose an $L^p$-space $\mathcal{E}$ and an isometric representation $\rho^+\colon A^+\rightarrow\mathcal{B}(\mathcal{E})$, and let $\rho=\rho^+|_A$. By \autoref{unitalrep}, we may choose $\rho^+$ to be unital, so that it preserves hermitian elements. Then $(\rho(S),\rho(T),\rho(E))$ is a Cuntz-Krieger $Q$-family in $\mathcal{B}(\mathcal{E})$. Define a spatial representation $\varphi\colon L_Q\rightarrow\mathcal{B}(\mathcal{E})$ by $\varphi(e_v)= \rho(E_v),\ \varphi(s_a)= \rho(S_a)$ and $\varphi(t_a)=\rho(T_a)$ for all $v\in Q^0$ and for all $a\in Q^1$. Extend $\varphi$ continuously to a representation of $\mathcal{O}^p(Q)$. As this representation is contractive, we have $\varphi(\mathcal{O}^p(Q))\subseteq\overline{\text{span}}\{\rho(S),\rho(T),\rho(E)\}\subseteq\text{Im}(\rho)$. Then $\rho^{-1}\circ\varphi\colon \mathcal{O}^p(Q)\rightarrow A$ is the desired contractive homomorphism.
\end{proof}

\begin{rem}
    Let $p\in[1,\infty)$ and let $(X,\mathcal{B},\mu)$ be a localizable measure space. If $\varphi\colon\mathcal{O}^p(Q)\rightarrow \mathcal{B}(L^p(\mu))$ is a nondegenerate representation, then $\varphi^+\colon\mathcal{O}^p(Q)^+\rightarrow \mathcal{B}(L^p(\mu))$ is unital and contractive and so preserves hermitian elements. Hence, $\varphi$ restricts to a spatial representation of $L_Q$. This provides a partial converse to the above proposition. However, the assumption of nondegeneracy is necessary, as shown in \cite[Remark 6.5]{cortinas rodriguez}.
\end{rem}

\begin{cor}\label{uniquenessiso}
    Let $p\in[1,\infty)$, let $Q$ be a graph and let $A$ be an $L^p$-operator algebra generated by a Cuntz-Krieger $Q$-family $(S,T,E)$. If $A$ has the universal property described in \autoref{opquniversal}, then the canonical homomorphism $\pi\colon \mathcal{O}^p(Q)\rightarrow A$ is an isometric isomorphism.
\end{cor}

\begin{proof}
    By \autoref{stespatial} and \autoref{opquniversal}, there are contractive homomorphisms $\pi\colon\mathcal{O}^p(Q)\rightarrow A$ and $\varphi\colon A\rightarrow\mathcal{O}^p(Q)$ such that $\pi(s_a)=S_a$ and $\varphi(S_a)=s_a$, and similarly for the $t_a$ and the $e_v$. Clearly the images of $\pi$ and $\varphi$ are dense. Since $\varphi\circ\pi$ and the identity map agree on generators, the identity $\varphi\circ\pi=\text{id}_{\mathcal{O}^p(Q)}$ holds by continuity on $\text{Im}(\varphi\circ\pi)\subseteq\mathcal{O}^p(Q)$. Since both $\varphi$ and $\pi$ are contractive, we see that $\pi$ is isometric on the dense subset $\text{Im}(\varphi\circ\pi)$, and hence on $\mathcal{O}^p(Q)$.
\end{proof}

\section{Spatial $L^p$-AF-algebras and acyclic graphs}

In this section, we recall the definition of a spatial $L^p$-AF-algebras (as defined by Phillips and Viola in \cite[Section 9]{phillips viola}), and show that contractive, injective homomorphisms out of such algebras are automatically isometric. 
Most of the work goes into proving that $\mathcal{O}^p(Q)$ is a spatial AF-algebra whenever $Q$ is acyclic; see \autoref{acyclicaf}. 

As we will see in the following section, the $L^p$-operator graph algebra of a regular graph with no 
infinite emitters and no sinks contains a 
distinguished subalgebra which is spatially AF, namely, the fixed point algebra associated to its gauge action; see \autoref{isocore}.

\subsection{Infinite direct sums of $L^p$-operator algebras}
In this subsection, we collect some elementary facts about infinite direct sums of 
$L^p$-operator algebras. The following definition is standard.

\begin{df}
    Let $p\in[1,\infty)$, and let $(A_i)_{i\in I}$ be a family of $L^p$-operator algebras. The \emph{direct sum} $\bigoplus_{i\in I} A_i$ is the set of elements $(a_i)_{i\in I}\in\prod_{i\in I} A_i$ satisfying that for any $\epsilon>0$ there exists a finite subset $J\subseteq I$ such that $\|a_i\|<\epsilon$ whenever $i\in I\setminus J$.
    The direct sum is equipped with componentwise operations and with norm $\|(a_i)_{i\in I}\|=\sup_{i\in I}\|a_i\|$.
\end{df}

In the proof of the following lemma, we will use \cite[Lemma 6.15]{phillips viola} in an algebra $A$
which is not necessarily separable. For this, we observe that the proof of \cite[Lemma 6.15]{phillips viola} also goes through if $A$ is not assumed to be separable\footnote{The role of separability in the work of Phillips-Viola is to allow one to assume that the $L^p$-operator algebra in question can be represented on a \emph{$\sigma$-finite} measure space, so that Lamperti's theorem for $\sigma$-finite measures can be applied. However, since Lamperti's theorem also holds for localizable measures, and \emph{every} $L^p$-space can be realized by a localizable measure, this allows us to remove the separability assumption.}, by applying \autoref{hermitianequiv} instead of \cite[Lemma 6.9]{phillips viola}.

\begin{lma}\label{directsum}
    Let $p\in[1,\infty)$ and let $A$ be a unitizable $L^p$-operator algebra. Let $(e_i)_{i\in I}$ be a family of pairwise orthogonal hermitian idempotents in $A$, and set $A_i=e_iAe_e$ for all $i\in I$. Then $\overline{\sum_{i\in I} A_i}$ is isometrically isomorphic to $\bigoplus_{i\in I} e_iAe_i$.
\end{lma}

\begin{proof}
Set $A_e=\overline{\sum_{i\in i} A_i}$, and let $\mathcal{S}$ be the set of finite subsets of $I$.
For $S\in\mathcal{S}$, set $B_S=\bigoplus_{i\in S} A_i$ and set $C_S=\sum_{i\in S} A_i$. Then $C_S$ is a unital $L^p$-operator algebra with $1_{C_S}=\sum_{i\in S}e_i$. By \cite[Lemma 2.2]{choi gardella thiel}, for $i\in S$ the idempotent $e_i$ is hermitian in $C_S$. Moreover, since these idempotents are pairwise orthogonal, each $e_i$ belongs to the center of $C_S$. By \cite[Lemma 6.15]{phillips viola}, the map $\phi_S\colon B_S\rightarrow C_S$ given by $(a_j)_{j\in S}^n\mapsto\sum_{j\in S}^n a_j$ is an isometric isomorphism. Let $\iota_S$ be the inclusion $C_S\hookrightarrow A_e$ and set $\psi_S=\iota_S\circ\phi_S$. Then also $\psi_S$ is an isometry. We have $\bigoplus_{i\in I} A_n=\overline{\bigcup_{S\in\mathcal{S}}B_S}$ and $B_R\subseteq B_S$ whenever $R\subseteq S$. As $\psi_S|_{B_R}=\psi_R$ in this case, we can define an isometric homomorphism $\psi\colon\bigoplus_{i\in I}A_i\rightarrow A_e$ by $\psi|_{B_S}=\psi_S$ for all $S\in\mathcal{S}$. Since $\text{Im}(\psi)$ is dense in $A_e$, it follows that $\psi$ is an isometric isomorphism.
\end{proof}

The following corollary is immediate, and is stated for reference.

\begin{cor}\label{homdirectsum}
    Let $p\in[1,\infty)$, let $(A_i)_{i\in I}$ be a family of unital $L^p$-operator algebras, and let $B$ be a unitizable $L^p$-operator algebra. Set $A=\bigoplus_{i\in I} A_n$, let $\iota_i\colon A_i\rightarrow A$ be the inclusion map and let $\varphi\colon A\rightarrow B$ be a contractive homomorphism. Suppose that the idempotent $e_i=\varphi(\iota_i(1_{A_i}))$ is hermitian in $B$ and that the map $\varphi\circ\iota_i$ is an isometry for all $i\in I$. Then $\varphi$ is an isometry.
\end{cor}

\begin{proof}
    For $i\in I$, set $B_i=\varphi(\iota_i(A_i))=e_i \varphi(A)e_i$. By \autoref{directsum}, we obtain a decomposition $\overline{\varphi(A)}=\bigoplus_{i\in I} B_i$. This decomposition exhibits $\varphi$ as a diagonal operator, which implies that it is isometric.
\end{proof}

\subsection{Spatial AF-algebras}
For $n\in\N$, we define $\ell^p_n=\ell^p(\{1,\dots,n\})$. Then the algebra $M_n$ of $n\times n$ matrices is canonically isomorphic, as a complex algebra, to $\B(\ell^p_n)$. The norm on $M_n$ induced by this identification is called the \emph{spatial ($L^p$-operator) norm} on $M_n$. The Banach algebra $M_n$ endowed with this norm will be denoted by $M_n^p$. For $j,k=1,\dots,n$, we let $e_{j,k}$ denote the corresponding matrix unit in $M_n^p$. Note that $e_{j,j}=m_{\mathbbm{1}_{\{j\}}}$ is a spatial idempotent, and that $e_{j,k}$ is a spatial partial isometry for each $j,k=1,\dots,n$. 

By a result of Phillips (\cite[Theorem 7.2]{phillips}), the spatial norm on $M_n$ is in fact minimal among all $L^p$-operator norms, in the sense that any nonzero, contractive homomorphism from $M_n^p$ into an $L^p$-operator 
algebra is automatically isometric. We isolate a consequence of this in our setting in the 
following lemma.

\begin{lma}\label{matrixiso}
Let $p\in[1,\infty)$, let $Q$ be a graph, let $n\in\N$ and let $\alpha_1,\dots,\alpha_n$ be paths in $Q$ such that $\{s_{\alpha_j}t_{\alpha_k}\}_{1\leq j,k\leq n}$ is a family of matrix units in $\mathcal{O}^p(Q)$. Then the homomorphism $\eta\colon M_n^p\rightarrow \text{span}\{s_{\alpha_j}t_{\alpha_k}\}_{1\leq j,k\leq n}$ determined by $\eta(e_{j,k})= s_{\alpha_j} t_{\alpha_k}$, for all $j,k=1,\ldots,n$, is an isometry.
\end{lma}

\begin{proof}
It is clear that the formula above defines an 
algebra homomorphism $\eta\colon M_n \to L_Q$. 
By \autoref{opqnormproperties}, there exist an $L^p$-space $\mathcal{E}$ and a spatial, nondegenerate representation $\varphi\colon L_Q\rightarrow\mathcal{B}(\mathcal{E})$ which extends to an isometric representation of $\mathcal{O}^p(Q)$. Set $\rho=\varphi\circ\eta\colon M_n\to \mathcal{B}(\mathcal{E})$.
We need to show that the spatial norm on 
$M_n$ agrees with the norm induced by $\rho$. 
By the equivalence between (2) and (3) in \cite[Theorem 7.2]{phillips}, it suffices
to show that for each $j,k=1,\ldots,n$, the element $\rho(e_{j,k})$ is a spatial partial isometry with reverse $\rho(e_{k,j})$. Since $\rho(e_{j,k})=\varphi(s_{\alpha_j}t_{\alpha_k})$ and $\varphi$ is a spatial representation of $L_Q$, the result follows.
\end{proof}

We now recall the definition of a spatial $L^p$-AF-operator algebra from \cite{phillips viola}, in a slightly more general form (since we allow arbitrary direct limits, not just over the natural numbers). For $p=2$, the definition agrees with the usual notion of an AF-algebra. 

\begin{df}
    Let $p\in[1,\infty)$. A \emph{spatial $L^p$-AF-direct system} is a direct system $((A_i)_{i\in I},(\varphi_{j,i})_{i\leq j})$ where for every $i\in I$, the algebra $A_i$ is a finite direct sum of matrix algebras endowed with the spatial norm, and for all $i,j\in I$ with $i\leq j$, the homomorphism $\varphi_{j,i}\colon A_i\rightarrow A_j$ is contractive and takes the identity of each of the matrix algebras constituting $A_i$ to a hermitian idempotent in $A_j$. A \emph{spatial $L^p$-AF-algebra} is a Banach algebra which is isometrically isomorphic to the direct limit of a spatial $L^p$-AF-direct system. 
\end{df}

It is well known that an injective, contractive homomorphism between C*-algebras is automatically isometric. For general $L^p$-operator algebras, the corresponding statement is false, and this is related to the lack of uniqueness of $L^p$-operator norms, for $p\neq 2$. 
This phenomenon makes the analytical study of $L^p$-operator algebras often very complicated: for example, we do not have a structure result for finite-dimensional $L^p$-operator algebras.
The following notion, originally due to Phillips, is an 
important property for $L^p$-operator algebras: 

\begin{df}(Phillips.)
    Let $p\in [1,\infty)$ and let $A$ be a Banach algebra. We say that $A$ is \emph{$p$-incompressible} if for every $L^p$-space $\mathcal{E}$, every contractive, injective homomorphism $\varphi\colon A\rightarrow\mathcal{B}(\mathcal{E})$ is isometric.
\end{df}

An important property of spatial $L^p$-AF-algebras is that they are $p$-incompressible; see \autoref{afincompressible} below. This is probably known to the experts, but does not appear explicitly in the literature even for spatial matrix 
algebras, so we include the argument.

\begin{lma}\label{afincompressible}
    Let $p\in [1,\infty)$, and let $A$ be spatial $L^p$-AF-algebra. Then $A$ is $p$-incompressible. 
\end{lma}

\begin{proof}
We first prove the statement assuming that $A$ is finite-dimensional.
Let $n,d_1,\ldots,d_n \in \N$ and suppose that $A\cong \bigoplus_{i=1}^n M_{d_i}^p$. Let $(X,\mathcal{B},\mu)$ be a measure space and let $\varphi\colon A\rightarrow \mathcal{B}(L^p(\mu))$ be a contractive injective homomorphism. Since $A$ is unital, by \autoref{unitalrep} we may assume that $\varphi$ is unital. Note that $\varphi(A)$ is finite-dimensional since $A$ is. Thus, $\varphi(A)$ is a closed and separable subalgebra of $\mathcal{B}(L^p(\mu))$, so it is a unital separable $L^p$-operator algebra. By \cite[Proposition 2.9]{phillips viola} there exist a localizable measure space $(Y,\mathcal{C},\nu)$ and an isometric unital representation $\rho\colon \varphi(A)\rightarrow \mathcal{B}(L^p(\nu))$. Hence we may assume that $\mu$ is localizable. 

Let $\iota_i\colon M_{d_i}^p\rightarrow A$ be the canonical inclusion. By \cite[Theorem~7.2]{phillips} and since $\mu$ is 
localizable, it follows that the composition $\varphi\circ\iota_j$ is an isometry for all $j=1,\dots,n$.
As $\varphi(1)=1$ is a hermitian idempotent, by \cite[Lemma 8.14]{phillips viola} the idempotent $\varphi(\iota_i(1_{M_{d_i}^p}))$ is hermitian for each $i=1,\dots,n$. It thus follows from \autoref{homdirectsum} that $\varphi$ is an isometry.

We now turn to the general case. By an argument almost identical to \cite[Proposition 9.11]{phillips viola}, any spatial $L^p$-AF-algebra is isometrically isomorphic to the direct limit of a spatial $L^p$-AF-direct system in which all the connecting maps are injective. By the first part of this proof, any finite direct sum of spatial matrix algebras is $p$-incompressible, and hence the result follows from \cite[Lemma 4.2]{gardella thiel}.
\end{proof}

\subsection{Spatial AF-structure in acyclic graphs}
Just as in the case of graph $C^*$-algebras, it turns out that if $Q$ is an acyclic graph, then $\mathcal{O}^p(Q)$ is a spatial $L^p$-AF-algebra. We will prove this in \autoref{acyclicaf}, by combining methods from \cite{cortinas rodriguez} and \cite[Section 2]{kumjian pask raeburn}.

We begin by showing that finite, acyclic graphs give rise to finite direct sums of spatial matrix algebras.

\begin{prop}\label{finacyclicsemisimple}
Let $p\in[1,\infty)$, let $Q$ be a finite acyclic graph, let $v\in Q^0$ be a source and set $F_v=\{\alpha\in Q^*\colon \sor(\alpha)=v\}$. Then:
\be \item $I_v=\text{span}\{s_\alpha t_\beta\colon \alpha,\beta\in F_v\}$ is a two-sided ideal in $\mathcal{O}^p(Q)$.
\item For $n_v=|F_v|$, there is an isometric isomorphism $I_v\cong M_{n_v}^p$.
\item Denote by $\mathcal{S}\subseteq Q^0$ the (finite and nonempty) set of all sources of $Q$. Then
\[\mathcal{O}^p(Q)\cong \bigoplus_{v\in\mathcal{S}} I_{v}\cong \bigoplus_{v\in\mathcal{S}} M_{n_v}^p.\]
\ee
\end{prop}

\begin{proof}
(1). Let $\alpha\in F_v$. As $v$ is a source, there are no paths of the form $\alpha\beta'$ with $|\beta'|>0$. Thus, recalling the relations in \autoref{tsrelationcomparable}, if $\alpha,\beta\in F_v$ and $\gamma,\delta\in Q^*$ we have
$$(s_\alpha t_\beta) (s_\gamma t_\delta)=
\begin{cases}
s_\alpha t_\delta &\text{if }\beta=\gamma,\\
        s_\alpha t_{\delta\beta'} &\text{if }\beta=\gamma\beta',\\
        0 &\text{else.}
\end{cases}$$
It follows that $(s_\alpha t_\beta) (s_\gamma t_\delta)$ belongs to $I_v$. This implies that for $a\in I_v$ and for $b\in L_Q$ we have $ab\in I_v$. Similar calculations show that $ba\in I_v$, and hence $I_v$ is a closed, two-sided ideal in $\mathcal{O}^p(Q)$.

(2). For $\alpha,\beta,\gamma,\delta\in F_v$, we have
    \begin{equation}\tag{$5.1$}\label{eqn:sourcepaths}
    (s_\alpha t_\beta)(s_\gamma t_\delta)=
            \begin{cases}
            s_\alpha t_\delta &\text{if }\beta=\gamma,\\
            0 &\text{else}.
        \end{cases}
    \end{equation}
    
Since $Q$ is finite and acyclic, the set $F_v$ is finite, and thus $\{s_\alpha t_\beta\colon \alpha,\beta\in F_v\}$ is a finite family of matrix units in $\mathcal{O}^p(Q)$. Now \autoref{matrixiso} gives the desired isometric isomorphism.

(3). Let $w\in Q^0$ and let $\alpha,\beta\in Q^*$ with $\sor(\alpha)=\sor(\beta)=w$. Suppose that $w$ is not a source. Then $s_\alpha t_\beta=\sum_{a\in \ran^{-1}(w)}s_\alpha s_a t_a t_\beta$, and for each summand either $s(a)$ is a source or we can repeat the process above. Since $Q$ is acyclic and finite, this must eventually lead to a representation of $s_\alpha t_\beta$ as a sum of terms of the form $s_{\alpha\gamma}t_{\beta\delta}$, where $\sor(\gamma)=\sor(\delta)=v$ for some $v\in \mathcal{S}$. Hence the ideals $I_{v}$, for $v\in\mathcal{S}$, span $\mathcal{O}^p(Q)$.

For $v\in \mathcal{S}$, set $\varepsilon_{v}=\sum_{\alpha\in F_{v}}s_\alpha t_\alpha$. Then $\varepsilon_{v}$ is hermitian, and it is an idempotent by (\ref{eqn:sourcepaths}). We note that if $\alpha\in F_{v}$ and $\beta\in F_{w}$ with $w\in \mathcal{S}\setminus\{v\}$, then $t_\alpha s_\beta=0$. It follows that the idempotents $\varepsilon_{v}$, for $v\in\mathcal{S}$, are pairwise orthogonal. Moreover, $\varepsilon_{v}\mathcal{O}^p(Q)\varepsilon_{v}=I_{v}$ for each $v\in\mathcal{S}$. Thus by \autoref{directsum}, we have $\mathcal{O}^p(Q)\cong\bigoplus_{v\in\mathcal{S}} I_{v}$. The last isomorphism follows from part~(2).
\end{proof}

\begin{lma}\label{finitegraphembedding}
    Let $p\in[1,\infty)$, let $Q$ be a graph and let $R$ be a finite subgraph of $Q$. Then there exists a finite graph $\overline{R}$ such that $R$ is a subgraph of $\overline{R}$ and such that there exists a Cuntz-Krieger $\overline{R}$-family in $\mathcal{O}^p(Q)$. Moreover, if $Q$ is acyclic then $\overline{R}$ can be chosen to be acyclic.
\end{lma}

\begin{proof}
We follow the constructions in \cite[Theorem 1.5.8]{lpa}. 
We write $\sor_Q$ and $\ran_Q$ for the structure maps of $Q$, and we write $\sor_R$ 
and $\ran_R$ for the maps associated to $R$.
Let $\mathfrak{R}$ be the finite graph obtained in the following way. For each vertex $v$ which is regular in $Q$ and which satisfies $\ran^{-1}_R(v)\neq\emptyset$, add the set $\ran^{-1}_Q(v)$ to $R^1$, and for all $a\in \ran^{-1}_Q(v)$ add $\sor_Q(a)$ to $R^0$. Now set 
    \[Y=\text{Reg}(\mathfrak{R})\setminus (\text{Reg}(\mathfrak{R})\cap\text{Reg}(Q))\] 
and let $Y'$ be a disjoint copy of $Y$. For $v\in Y$, let $v'$ denote the corresponding element of $Y'$, and for each $a\in \ran^{-1}_{\mathfrak{R}}(v)$ introduce a symbol $a'$. We let $\overline{R}$ be the finite graph with 
\[\overline{R}^0=\mathfrak{R}^0\sqcup Y'\text{ and }\overline{R}^1=\mathfrak{R}^1\sqcup\{a'\colon \ran^{-1}(a)\in Y\},\] 
    and we set $\sor_{\overline{R}}(a')=\sor_{\mathfrak{R}}(a)'$ and $\ran_{\overline{R}}(a')=\ran_{\mathfrak{R}}(a)$. Note that $\mathfrak{R}$ is acyclic if $Q$ is acyclic. Moreover, since $\sor_{\overline{R}}(a)$ is a source for each $a\in\overline{R}^1\setminus\mathfrak{R}^1$, also $\overline{R}$ is acyclic in this case.

    For $v\in\text{Reg}(\mathfrak{R})$, set $\varepsilon_v=\sum_{\ran^{-1}_{\mathfrak{R}}(v)}s_at_a\in L_Q$. Then each $\varepsilon_v$ is an idempotent. With the goal of defining a Cuntz-Krieger $\overline{R}$-family $(S,T,E)$ in $\mathcal{O}^p(Q)$, we define elements in $L_Q$ as follows:

    \begin{itemize}
        \item if $v\notin Y$, set $E_v=e_v$,
        \item if $v\in Y$, set $E_v=\varepsilon_v$ and $E_{v'}= e_v-\varepsilon_v$,
        \item if $\sor_Q(a)\notin Y$, set $S_a= s_a$ and $T_a= t_a$,
        \item if $\sor_Q(a)\in Y$, set $S_a= s_a \varepsilon_{\sor_Q(a)}$ and $T_a= \varepsilon_{\sor_Q(a)} t_a$, and set $S_{a'}= s_a (e_{\sor_Q(a)}-\varepsilon_{\sor_Q(a)})$ and $T_{a'}= (e_{\sor_Q(a)}-\varepsilon_{\sor_Q(a)}) t_a$.
    \end{itemize}
    Our goal is to show that $(S,T,E)$ is a Cuntz-Krieger $\overline{R}$-family in $\mathcal{O}^p(Q)$. In other words, we need to check that:
    \begin{enumerate}
        \item The $E_v$, for $v\in \overline{R}^0$, are pairwise orthogonal idempotents and $S_a$, for $a\in \overline{R}^1$, is a partial isometry with reverse $T_a$.
        \item The relations in \autoref{spatialckfam} are satisfied.
        \item $E_v$ is hermitian for all $v\in \overline{R}^0$.
        \item $S_aT_a$ is hermitian for all $a\in \overline{R}^1$.
        \item We have $\|S_a\|\leq 1$ and $\|T_a\|\leq 1$ for all $a\in \overline{R}^1$. 
    \end{enumerate}
    
We note a few properties of the elements $\varepsilon_v$. We have $\varepsilon_ve_v=e_v\varepsilon_v=\varepsilon_v$ for all $v\in \text{Reg}(\mathfrak{R})$, and if $w\in\mathfrak{R}^0$ with $w\neq v$ we have $e_v\varepsilon_w=\varepsilon_we_v=\varepsilon_v\varepsilon_w=0$. Lastly if $a\in\mathfrak{R}^1$ with $\ran_Q(a)\in\text{Reg}(\mathfrak{R})$, then since $s_at_a$ is a summand of $\varepsilon_{\ran_Q(a)}$ we have $\varepsilon_{\ran_Q(a)}s_a=s_a$ and $t_a\varepsilon_{\ran_Q(a)}=t_a$.
    
The first two properties imply that the $E_v$ are pairwise orthogonal idempotents for $v\in\overline{R}^0$. If $a\in\mathfrak{R}^1$ with $\sor_Q(a)\in Y$, then 
$$S_aT_aS_a=s_a \varepsilon_{\sor_Q(a)}\varepsilon_{\sor_Q(a)}t_a s_a \varepsilon_{\sor_Q(a)}=s_a\varepsilon_{\sor_Q(a)}e_{\sor_Q(a)}\varepsilon_{\sor_Q(a)}=s_a\varepsilon_{\sor_Q(a)}=S_a.$$ 
In a similar fashion one shows the remaining identities needed for (1).

We proceed to (2), and begin by showing that $E_{\ran_Q(a)}S_a=S_a=S_a E_{\sor_Q(a)}$ for all $a\in\overline{R}^1$. As $S_a$ is of the form $s_ax$ for some $x\in L_Q$, and as $E_{\ran_Q(a)}$ either equals $e_{\ran_Q(a)}$ or $\varepsilon_{\ran_Q(a)}$, both of which act trivially on the left on $s_a$, we see that $E_{\ran_Q(a)}S_a=S_a$ for all $a\in\overline{R}^1$. Clearly, if $\sor_Q(a)\notin Y$ then $S_aE_{\sor_Q(a)}=S_a$. If $\sor_Q(a)\in Y$, then $S_aE_{\sor_Q(a)}=s_a \varepsilon_{\sor_Q(a)}\varepsilon_{\sor_Q(a)} =S_a$, and the case of $S_{a'}E_{\sor_Q(a')}=S_{a'}$ is similar. The corresponding relations for the $T_a$ are shown in a similar way.

We must show that $T_aS_b=\delta_{a,b}E_{\sor_Q(a)}$ for all $a,b\in\overline{R}^1$. Clearly this holds if both $\sor_Q(a),\sor_Q(b)\notin Y$. If only one of $\sor_Q(a)$ and $\sor_Q(b)$ belong to $Y$, say $\sor_Q(a)$, then $t_as_b=0$ which implies that $T_aS_b=T_{a'}S_b=0$. If both $\sor_Q(a),\sor_Q(b)\in Y$ then 
$$T_aS_b=\varepsilon_{\sor_Q(a)} t_as_a \varepsilon_{\sor_Q(a)}=\varepsilon_{\sor_Q(a)}\delta_{a,b}e_{\sor_Q(a)}\varepsilon_{\sor_Q(a)}=\delta_{a,b}\varepsilon_{\sor_Q(a)}=\delta_{a,b}E_{\sor_Q(a)},$$ 
and in a similar way one shows that $T_{a'}S_{b'}=\delta_{a,b}E_{\sor_Q(a')}$. We must also show that $T_{a'}S_b=T_aS_{b'}=0$. We have 
$$T_{a'}S_b=(e_{\sor_Q(a)}-\varepsilon_{\sor_Q(a)}) t_as_b \varepsilon_{\sor_Q(b)}=\delta_{a,b}(e_{\sor_Q(a)}-\varepsilon_{\sor_Q(a)})\varepsilon_{\sor_Q(a)}=0,$$ 
and the other case is similar. 
    
The last relation in (2) is $E_v=\sum_{\ran^{-1}_{\overline{R}}(v)}S_aT_a$ for all $v\in\text{Reg}(\overline{R})$. Suppose $v\notin Y$, so that $v$ is regular in $Q$ and $\ran^{-1}_{\overline{R}}(v)=\ran^{-1}_Q(v)$. We write the right hand side as 
\begin{align*}
\sum_{\substack{\ran_Q(a)=v \\ \sor_Q(a)\notin Y}} S_aT_a &+\sum_{\substack{\ran_Q(a)=v \\ \sor_Q(a)\in Y}} S_aT_a+\sum_{\substack{\ran_Q(a)=v \\ \sor_Q(a)\in Y}}S_{a'}T_{a'} \\
&= \sum_{\substack{\ran_Q(a)=v \\ \sor_Q(a)\notin Y}} s_at_a + \sum_{\substack{\ran_Q(a)=v \\ \sor_Q(a)\in Y}}s_a \varepsilon_{\sor_Q(a)} t_a+\sum_{\substack{\ran_Q(a)=v \\ \sor_Q(a)\in Y}}s_a(e_{\sor_Q(a)}-\varepsilon_{\sor_Q(a)})t_a\\
&=\sum_{\substack{\ran_Q(a)=v \\ \sor_Q(a)\notin Y}} s_at_a+\sum_{\substack{\ran_Q(a)=v \\ \sor_Q(a)\in Y}} s_at_a =e_v=E_v.
\end{align*}
Now if $v\in Y$, the same calculation as above gives
\begin{align*}
\sum_{a\in \ran^{-1}_{\overline{R}}(v)}S_aT_a=\sum_{a\in \ran^{-1}_{\mathfrak{R}}(v)}s_at_a=\varepsilon_v=E_v.
\end{align*}
Noting that none of the $v'$ are regular in $\overline{R}$, since $\ran_Q^{-1}(v')=\emptyset$ for all $v$, we are done.

The $E_v$ and the $E_{v'}$ are hermitian as they are finite sums of hermitian elements
(see \autoref{preservehermitian}). Then they are contractive, so all $S_a,S_{a'},T_a,T_{a'}$ are also contractive. Also the $S_aT_a$ and the $S_{a'}T_{a'}$ are hermitian, which can be seen by considering the isometric representation $\varphi$ of $\mathcal{O}^p(Q)$ from \autoref{opqnormproperties}. As the set of spatial partial isometries is a semigroup, the image of these idempotents under $\varphi$ are spatial idempotents, and as $\varphi^{-1}$ preserves hermitian elements (since it is isometric and unital), these idempotents are indeed hermitian. This finishes the proof that $(S,T,E)$ is a Cuntz-Krieger $\overline{R}$-family in $\mathcal{O}^p(Q)$.
\end{proof}

\begin{thm}\label{acyclicaf}
    Let $p\in[1,\infty)\setminus\{2\}$, and let $Q$ be an acyclic graph. Then $\mathcal{O}^p(Q)$ is a spatial $L^p$-AF-algebra.
\end{thm}

\begin{proof}
    By \autoref{finitegraphembedding} for each finite subgraph $R$ of $Q$ there exist a finite acyclic graph $\overline{R}$ containing $R$ and a Cuntz-Krieger $\overline{R}$-family in $\mathcal{O}^p(Q)$. Let 
    \[\pi_{\overline{R}}\colon \mathcal{O}^p(\overline{R})\rightarrow\mathcal{O}^p(Q)\] 
    be the canonical contractive homomorphism of \autoref{opquniversal}. By \autoref{finacyclicsemisimple}, $\mathcal{O}^p(\overline{R})$ is a finite direct sum of spatial matrix algebras, and by \cite[Lemma 8.10]{phillips viola}, so is $\mathcal{O}^p(\overline{R})/\text{ker}(\pi_{\overline{R}})$. Hence, by \autoref{afincompressible} the quotient is $p$-incompressible. As the induced map $\widehat{\pi_{\overline{R}}}\colon \mathcal{O}^p(\overline{R})/\text{ker}(\pi_{\overline{R}})\rightarrow \pi_{\overline{R}}(\mathcal{O}^p(\overline{R}))$ is a contractive bijective homomorphism, it is an isometric isomorphism. Hence also $\pi_{\overline{R}}(\mathcal{O}^p(\overline{R}))$ is a finite direct sum of spatial matrix algebras. Let $\mathcal{I}$ be the set of finite subgraphs of $Q$, partially ordered by inclusion. If $R,S\in \mathcal{I}$ with $R\subseteq S$, let $\iota_{S,R}\colon \pi_{\overline{R}}(\mathcal{O}^p(\overline{R}))\hookrightarrow\pi_{\overline{S}}(\mathcal{O}^p(\overline{S}))$ be the inclusion. Then 
    \[\big((\pi_{\overline{R}}(\mathcal{O}^p(\overline{R}))_{R\in\mathcal{I}},(\iota_{S,R})_{R\subseteq S}\big)\] is a spatial $L^p$-AF-direct system. 

    In order to show that $\mathcal{O}^p(Q)$ is the direct limit of this system, we want to show that each finite set of elements of $\mathcal{O}^p(Q)$ can be approximated by elements lying in $\pi_{\overline{R}}(\mathcal{O}^p(\overline{R}))$ for some $R\in\mathcal{I}$. As $L_Q$ is dense in $\mathcal{O}^p(Q)$, it is enough to show that any finite set of elements on the form $s_\alpha t_\beta$ for $\alpha,\beta\in Q^*$ lies in some $\pi_{\overline{R}}(\mathcal{O}^p(\overline{R}))$. Suppose therefore we are given a finite set $F$ of pairs of paths $(\alpha,\beta)$ with $\sor(\alpha)=\sor(\beta)$. Let $R$ be the finite subgraph of $Q$ consisting of all edges and vertices occurring in the paths $\alpha,\beta$ with $(\alpha,\beta)\in F$. Then 
    \[\{s_\alpha t_\beta\colon (\alpha,\beta)\in F\}\subseteq\text{span}\{s_\alpha t_\beta\colon \alpha,\beta\in \overline{R}^*\}=\pi_{\overline{R}}(\mathcal{O}^p(\overline{R})),\]
    as desired.
\end{proof}

\section{The gauge action and its spectral subspaces}

\begin{df}
    Let $G$ be a topological group and let $A$ be a Banach algebra. An \emph{action} of $G$ on $A$ is a homomorphism $\alpha\colon G\rightarrow\text{Aut}(A)$ such that for each $a\in A$, the map $\alpha^a\colon G\rightarrow A$ given by $\alpha^a(s)=\alpha_s(a)$ for all $s\in G$ is continuous. The action $\alpha$ is called \emph{isometric} if $\alpha_s$ is isometric for each $s\in G$.
\end{df}

We denote by $\mathbb{T}$ the unit circle in $\C$, regarded as a multiplicative group.

\begin{df}
    Let $A$ be a Banach algebra equipped with an action $\alpha$ of $\mathbb{T}$. For $n\in\Z$, we define the associated \emph{spectral subspace} $A^n\subseteq A$ by
    $$A^n=\{a\in A\colon \alpha_z(a)=z^n a\text{ for all }z\in\mathbb{T}\}.$$ 
    Note that $A^n$ is a closed linear subspace of $A$ for each $n\in\Z$. We have $A^m A^n\subseteq A^{m+n}$, so in particular $A^0$ is a Banach subalgebra of $A$. We write $A^\alpha=A^0$, and call it the \emph{fixed point algebra} of $\alpha$.
\end{df}

In the next proposition, we will use some standard facts about integration 
in Banach algebras, for example as in \cite[Proposition B.34]{williams}.
The only non-trivial claim is (4), which can be interpreted as saying that the 
family $(\Phi_n)_{n\in\Z}$ is (jointly) faithful. 

\begin{prop}\label{phindef}
    Let $A$ be a Banach algebra equipped with an isometric action $\alpha$ of $\mathbb{T}$. For $n\in\Z$, define $\Phi_n\colon A\rightarrow A$ by 
    $$\Phi_n(a)=\int_\mathbb{T}z^{-n}\alpha_z(a)dz$$
    for all $a\in A$. Then:
    \begin{enumerate}
        \item $\text{Im}(\Phi_n)\subseteq A^n$ for all $n\in\Z$.
        \item $\Phi_n$ agrees with the identity on $A^n$ for all $n\in\Z$.
        \item $\Phi_n$ is linear and contractive for all $n\in\Z$.
        \item If $a\in A$ satisfies $\Phi_n(a)=0$ for all $n\in\Z$, then $a=0$.
    \end{enumerate}
\end{prop}

\begin{proof}
    (1) Let $n\in\Z$, let $w\in\mathbb{T}$ and let $a\in A$. Since $\alpha_w$ and $\alpha$ are homomorphisms, we have
    \begin{align*}
        \alpha_w(\Phi_n(a))&=\int_\mathbb{T}z^{-n}\alpha_w(\alpha_z(a))dz=w^n\int_\mathbb{T}(wz)^{-n}\alpha_{wz}(a)dz\\
        &=w^n\int_\mathbb{T}z^{-n}\alpha_z(a)dz=w^n\Phi_n(a).
    \end{align*}
    Hence $\Phi_n(a)\in A^n$. 
    
    (2) and (3) are immediate.

    (4) Suppose that $a\in A$ satisfies $\Phi_n(a)=0$ for all $n\in\Z$, and let $\varepsilon>0$ be given. We want to show that $\|a\|<\varepsilon$. Use continuity of the assignment $z\mapsto\alpha_z(a)$ to find an open neighborhood $U$ of $1\in\mathbb{T}$ such that $\|\alpha_z(a)-a\|<\varepsilon$ for all $z\in U$. Choose a non-negative function $f\in C(\mathbb{T})$ such that $f|_{U^c}=0$ and $\int_\mathbb{T} f(z)dz=1$. The assumption on $a$ implies that $\int_\mathbb{T}q(z)\alpha_z(a)dz=0$ for all polynomials $q(z)$. As the polynomials are dense in $C(\mathbb{T})$, we see that $\int_\mathbb{T}g(z)\alpha_z(a)dz=0$ for all $g\in C(\mathbb{T})$. Hence 
    \begin{align*}
        \|a\|&=\Big\|a-\int_\mathbb{T}f(z)\alpha_z(a) dz\Big\|=\Big\|\int_\mathbb{T}f(z)(a-\alpha_z(a)) dz\Big\|\\
        &\leq \int_Uf(z)\|a-\alpha_z(a)\| dz<\varepsilon,   
    \end{align*}
    as we wanted to show.
\end{proof}

\begin{prop}\label{gaugeconstruct}
    Let $p\in[1,\infty)$ and let $Q$ be a graph. Then there exists an isometric action $\gamma$ of $\mathbb{T}$ on $\mathcal{O}^p(Q)$ satisfying $\gamma_z(s_a)=zs_a,\ \gamma_z(t_a)=\overline{z}t_a$ and $\gamma_z(e_v)=e_v$ for all $z\in\mathbb{T}$, for all $a\in Q^1$ and for all $v\in Q^0$.
\end{prop}

\begin{proof}
    This follows immediately from \autoref{opquniversal}, since for $z\in\mathbb{T}$, the family $(zS,\overline{z}T,E)$ is a Cuntz-Krieger $Q$-family whenever $(S,T,E)$ is.
\end{proof}

\begin{df}
    The action $\gamma\colon\mathbb{T}\rightarrow\text{Aut}(\mathcal{O}^p(Q))$ constructed in \autoref{gaugeconstruct} is called the \emph{gauge} action.
\end{df}

In the following, the only action of $\mathbb{T}$ on $\mathcal{O}^p(Q)$ we will consider is the gauge action $\gamma$. For $n\in\Z$, we let $\mathcal{O}^p(Q)^n$ and $\Phi_n$ denote the $n$-th spectral subspace and the $n$-th projection map, respectively, associated to $\gamma$.

\begin{lma}\label{corephi}
    Let $p\in[1,\infty)$ and let $Q$ be a graph. Let $\gamma\colon\mathbb{T}\rightarrow\text{Aut}(\mathcal{O}^p(Q))$ be the gauge action. For $n\in\Z$, we have 
    $$\mathcal{O}^p(Q)^n=\text{Im}(\Phi_n)=\overline{\text{span}}\{s_\alpha t_\beta\colon |\alpha|-|\beta|=n\}.$$
\end{lma}

\begin{proof}
    From (1) and (2) of \autoref{phindef} we have $\text{Im}(\Phi_n)\subseteq\mathcal{O}^p(Q)^n$ and $\Phi_n=\text{id}$ on $\mathcal{O}^p(Q)^n$, so that $\mathcal{O}^p(Q)^n\subseteq\text{Im}(\Phi_n)$. This proves the first equality. If $\alpha,\beta$ are paths with $|\alpha|-|\beta|=n$, for $z\in\mathbb{T}$ we have $\gamma_z(s_\alpha t_\beta)=z^ns_\alpha t_\beta$. As $\mathcal{O}^p(Q)^n$ is a closed subspace, we get $\overline{\text{span}}\{s_\alpha t_\beta\colon |\alpha|-|\beta|=n\}\subseteq\mathcal{O}^p(Q)^n$. Lastly, for arbitrary paths $\alpha$ and $\beta$ we have 
    $$\Phi_n(s_\alpha t_\beta)=\int_\mathbb{T}z^{-n}z^{|\alpha|-|\beta|}s_\alpha t_\beta dz=\delta_{n,|\alpha|-|\beta|}s_\alpha t_\beta.$$ 
    As $\Phi_n$ is linear and continuous, this proves the inclusion $\text{Im}(\Phi_n)\subseteq\overline{\text{span}}\{s_\alpha t_\beta\colon |\alpha|-|\beta|=n\}$. 
\end{proof}

Next, we introduce some notation that will be convenient when describing the 
spectral subspaces of the gauge action.

\begin{nota}
For a graph $Q$ and $k\in\Z$, we define
\begin{align*}
        \mathcal{F}_k&=\overline{\text{span}}\{s_\alpha t_\beta\colon \alpha,\beta\in Q^k\},\\
        \mathcal{F}_k(v)&=\overline{\text{span}}\{s_\alpha t_\beta\colon \alpha,\beta\in Q^k\cap \sor^{-1}(v)\}\text{ for }v\in Q^0.
\end{align*}
Note that $\mathcal{F}_k$ and $\mathcal{F}_k(v)$ are (closed) subalgebras of $\mathcal{O}^p(Q)$.
\end{nota}

From now on, we write $\N=\{0,1,2,\ldots\}$.

\begin{lma}\label{fkdirectsum}
    Let $p\in[1,\infty)$, let $Q$ be a graph with no infinite emitters and let $k\in\N$. Then $\mathcal{F}_k$ is isometrically isomorphic to $\bigoplus_{v\in Q^0}\mathcal{F}_k(v)$.
\end{lma}

\begin{proof}
    For $v\in Q^0$, define $\varepsilon_v=\sum_{\alpha\in Q^k\cap \sor^{-1}(v)}s_\alpha t_\alpha$. The sum is finite because $Q$ has no infinite emitters, and thus $\varepsilon_v$ is hermitian by \autoref{preservehermitian}. Moreover, we have $\varepsilon_v\mathcal{F}_k\varepsilon_v=\mathcal{F}_k(v)$. One can easily check that the $\varepsilon_v$, for $v\in Q^0$, are pairwise orthogonal idempotents. We have $\mathcal{F}_k=\overline{\sum_{v\in Q^0}\mathcal{F}_k(v)}$ by definition. Thus, the result follows by \autoref{directsum}. 
\end{proof}

We obtain the following useful result.

\begin{cor}\label{isocore}
Let $p\in[1,\infty)$ and let $Q$ be a regular graph with no 
infinite emitters and no sinks. Then $\mathcal{O}^p(Q)^\gamma$ is a spatial $L^p$-AF-algebra.

Moreover, if $(S,T,E)$ is a Cuntz-Krieger $Q$-family in a unitizable 
$L^p$-operator algebra $A$ with $E_v\neq 0$ for all $v\in Q^0$, then the canonical homomorphism $\pi\colon \mathcal{O}^p(Q)\rightarrow A$ is isometric on $\mathcal{O}^p(Q)^\gamma$.
\end{cor}
\begin{proof}
By \autoref{corephi} we have $\mathcal{O}^p(Q)^\gamma=\overline{\bigcup_{k\in\N}\mathcal{F}_k}$, and since $Q$ is regular, for $s_\alpha t_\beta\in\mathcal{F}_k(v)$ we have $s_\alpha t_\beta=\sum_{a\in \ran^{-1}(v)} s_{\alpha a} t_{\beta a}\in\mathcal{F}_{k+1}$, so that there is an isometric inclusion $\mathcal{F}_k\subseteq\mathcal{F}_{k+1}$. Moreover, for $k\in\N$, each $\mathcal{F}_k$ is a 
(possibly infinite) direct sum of spatial matrix algebras by \autoref{fkdirectsum} and \autoref{afincompressible},
so $\mathcal{F}_k$ is itself a spatial $L^p$-AF-algebra.

We turn to the second statement. The assumption that $E_v\neq 0$ for all $v\in Q^0$ implies 
that the restriction of $\pi$ to $\mathcal{O}^p(Q)^\gamma$ is injective, and hence the 
claim follows from \autoref{afincompressible}.
\end{proof}
 
We will need a stronger version of the result above for graphs without sinks, 
asserting that the canonical 
homomorphism $\pi$ from \autoref{isocore} is isometric not only on the fixed point
subalgebra but also on all spectral subspaces. Incompressibility of spatial 
$L^p$-AF-algebras is not useful in this setting, since the spectral subspaces are 
not subalgebras.

\begin{prop}\label{isometryspectral}
Let $p\in[1,\infty)$ and let $Q$ be a regular graph with no 
infinite emitters. Let $(S,T,E)$ be a Cuntz-Krieger $Q$-family in a unitizable $L^p$-operator algebra $A$ with $E_v\neq 0$ for all $v\in Q^0$, and let $n\in\Z$. 
Then the canonical homomorphism $\pi\colon \mathcal{O}^p(Q)\rightarrow A$ is isometric on $\mathcal{O}^p(Q)^n$.\end{prop}
\begin{proof}
Suppose $n\geq 0$. For $k\in\N$, set $\mathcal{F}_k^n=\overline{\text{span}}\{s_\alpha t_\beta\colon |\alpha|=n+k,\ |\beta|=k\}$. Note that we have $\mathcal{O}^p(Q)^n= \overline{\bigcup_{k\in\N}\mathcal{F}_k^n}$, and that the union is increasing. In particular, if $\pi$ is an isometry on each $\mathcal{F}_k^n$, then it is an isometry on $\mathcal{O}^p(Q)^n$. By continuity, it suffices to prove that $\pi$ is isometric on $\mathcal{F}_k^n\cap L_Q$ for any $k\in\N$. 
   
Fix $a\in \mathcal{F}_k^n\cap L_Q$. Find finite sets $F\subseteq Q^{k+n}$ and $G\subseteq Q^k$ and scalars $\lambda_{\alpha,\beta}\in \C$ for $(\alpha,\beta)\in F\times G$, such that $a=\sum_{\alpha\in F,\ \beta\in G} \lambda_{\alpha,\beta}s_\alpha t_\beta$. For each $v\in \ran^{-1}(G)$, use the fact that $Q$ has no sinks to choose a path $\tau_v$ with $\sor(\tau_v)=v$ and $|\tau_v|=n$. Set $x=\sum_{v\in \ran^{-1}(G)}t_{\tau_v}$ and $x^*=\sum_{v\in \ran^{-1}(G)}s_{\tau_v}$. Note that $ax\in\mathcal{O}^p(Q)^\gamma$. As the $\tau_v$ are distinct and of equal length, they are incomparable, and hence by \autoref{propertiesspatialfam} the $t_{\tau_v}$ have orthogonal ranges. Therefore we have $\|x\|=\|x^*\|=1$. By construction, we have $xx^*=\sum_{v\in \ran^{-1}(G)}e_v$, and hence $axx^*=a$. Thus, we have 
$$\|a\|=\|axx^*\|\leq \|ax\|\leq \|a\|,$$ 
so that $\|ax\|=\|a\|$. By a similar reasoning we have $\|\pi(ax)\|=\|\pi(a)\|$. Thus, by the second part of \autoref{isocore}, we have 
\[\|a\|=\|ax\|=\|\pi(ax)\|=\|\pi(a)\|,\]
as desired.

For $n<0$, one defines $\mathcal{F}_k^n=\overline{\text{span}}\{s_\alpha t_\beta\colon |\alpha|=k,\ |\beta|=k-n\}$ for $k\in\N$, and sets $x=\sum_{v\in \ran^{-1}(F)} s_{\tau_v}$, where $\tau_v\in Q^{|n|}\cap \sor^{-1}(v)$. Then $\|x\|=\|x^*\|=1,\ xa\in\mathcal{O}^p(Q)^\gamma$ and $x^*xa=a$, and one may argue as above. We omit the details.
\end{proof}

\section{Uniqueness theorems for $L^p$-graph algebras}

In this section, we prove our main results (\autoref{gauge-uniqueness} and \autoref{uniqueness}), which are $L^p$-versions of the gauge-invariant and the Cuntz-Krieger uniqueness theorems for graph C*-algebras (see, for example,
\cite[Theorems~2.2~and~2.4]{raeburn}). We will use the fact that any graph is a directed union of countable graphs with suitable properties, and so the main body of the proofs will be concerned with the case where $Q$ is countable. 

\subsection{Cuntz-Krieger subgraphs}
In this subsection, we show that the $L^p$-operator algebra of an arbitrary graph
can be written as the direct limit, with isometric homomorphisms, of the $L^p$-operator
algebras of countable subgraphs. For this, we will 
need to recall the following definition from \cite[Subsection 2.3]{goodearl}.

\begin{df}
A subgraph $R$ of a graph $Q$ is said to be a \emph{Cuntz-Krieger subgraph}
if for every $v\in R^0$ with $0<|\mathsf{r}_R^{-1}(v)|<\infty$, 
we have $\mathsf{r}_R^{-1}(v)=\mathsf{r}_Q^{-1}(v)$. 
\end{df}

Let $Q$ be a graph and let $R$ be any subgraph. Then there is a canonical 
\emph{linear} map $L_R\to L_Q$, sending each generator of $L_R$ to the ``same" 
element in $L_Q$. This map will in general not be multiplicative, since condition CK2
for some vertices in $R$ may fail in $Q$. The condition in the definition above
guarantees that $L_R\to L_Q$ is an algebra homomorphism. 

By \cite[Proposition 2.7]{goodearl}, any graph
$Q$ is the directed union of the set CK$(Q)$ of all its countable Cuntz-Krieger subgraphs, and hence $L_Q$ is the direct limit, in the category of complex algebras, of the $L_R$, for $R\in$ CK$(Q)$. Our next result will be used to show that $\mathcal{O}^p(Q)$ is the direct limit, in the category of Banach algebras with isometric homomorphism, of the $\mathcal{O}^p(R)$, for $R\in$ CK$(Q)$; 
see \autoref{cor:DirLim}.

Recall that a measure space $(X,\mathcal{A})$ is said to be \emph{standard Borel} if there is a separable, complete metric space $Z$ such that $(X,\mathcal{A})$ is 
isomorphic to $Z$ endowed with its Borel $\sigma$-algebra. If $\mu$ is a probability
measure on $(X,\mathcal{A})$, we call the triple $(X,\mathcal{A},\mu)$ a \emph{standard Borel probability space}. It is well-known that there is a unique (up to measure-preserving bi-measurable bijections) \emph{atomless} standard Borel probability space; a result that we will use repeatedly in the following proof. 

\begin{prop}\label{prop:DirLim}
Let $Q$ be a countable graph, let $R\subseteq Q$ be a Cuntz-Krieger subgraph, and let $p\in [1,\infty)$.
Then the canonical map $L_R\to L_Q$ extends to an isometric homomorphism 
$\mathcal{O}^p(R)\to \mathcal{O}^p(Q)$. 
\end{prop}
\begin{proof}
The case $p=2$ is well-known, so we will assume that $p\neq 2$. By \cite[Proposition~9.8]{phillips viola}, there exist 
a $\sigma$-finite measure space $(X,\mu)$ and an isometric, nondegenerate representation
$\varphi\colon \mathcal{O}^p(R)\to \mathcal{B}(L^p(X_0,\mu_0))$. Let $(S,T,E)=(\varphi(s),\varphi(t),\varphi(e))$ be the associated Cuntz-Krieger $R$-family.
Find a standard Borel probability space $(X,\mu)$ 
(which may or may not have atoms) for which there is an isometric isomorphism $L^p(X,\mu)\cong L^p(X_0,\mu_0)$. Using the Lebesgue measure on $[0,1]$, and replacing $X$ with $X\times [0,1]$ and $\varphi$ with $\varphi\otimes\id_{L^p([0,1])}$ (which induces the same norm as $\varphi$ on $\mathcal{O}^p(R)$), we may assume that $(X,\mu)$ is standard Borel and atomless; recall that there is a unique measure space with these properties. For $v\in R^0$, the support $X_v$ of the associated hermitian idempotent $E_v$ cannot be empty (since otherwise $\varphi$ would not be injective). If $\mu_v$ denotes the restriction of $\mu$ to $X_v$, by the argument just described we can 
assume that $(X_v,\mu_v)$ is itself also standard Borel and atomless.
In particular, and we will need this later, for any $v,w\in R^0$ there is a measure preserving isomorphism $(X_v,\mu_v)\cong (X_w,\mu_w)$.

We will now construct an atomless standard Borel probability space $(Y,\nu)$ and a Cuntz-Krieger $Q$-family $(S',T',E')$ on $(Y,\nu)$ such that the restriction to $L_R$ of the associated representation of $L_Q$
is an amplification of the original representation $\varphi$ on $(X,\mu)$. Once we 
accomplish this, it will follow that the restriction of the norm on $\mathcal{O}^p(Q)$ to $L_R$ agrees with the norm on $\mathcal{O}^p(R)$, thus proving the proposition.

To construct the space $Y$ and the representation of $L_Q$, we first classify the vertices of $Q$ into three categories. Any vertex $v\in Q^0$ satisfies precisely one of the following:
\begin{enumerate}
\item $v\in R^0$ is regular in $R$. In this case, we set $(Y_v,\nu_v)=(X_v,\mu_v)$
\item $v\in R^0$ is singular in $R$. In this case, we set 
\[(Y_v,\nu_v)=\big(X_v\times \mathsf{r}_Q^{-1}(v),\mu_v\times \mathrm{counting}\big).\]
\item $v\notin R^0$. In this case, we set 
\[(Y_v,\nu_v)=\big([0,1]\times\mathsf{r}_Q^{-1}(v),\mathrm{Lebesgue}\times\mathrm{counting}\big).\]
\end{enumerate}	
Note that $(Y_v,\nu_v)$ is standard Borel and atomless in all cases. We let $E'_v$
denote the hermitian idempotent associated to the set $Y_v$.

Set $(Y,\nu)=\bigsqcup_{v\in Q^0}(Y_v,\nu_v)$. For $a\in Q^1$, we will define $S_a'$ according to the cases we describe below; in all instances, $T_a'$ will denote the transpose of $S_a'$. Fix $a\in Q^1$, and set $\mathsf{s}_Q(a)=w$ and $\mathsf{r}_Q(a)=v$. We enumerate all possible cases using the classification of vertices into the cases (1), (2) and (3) described above\footnote{Some of these cases could be grouped together, but we describe them all explicitly for clarity.}:

\begin{enumerate}
	\item[(1)$\to$(1):] $w\in R^0$ is regular in $R$ and $v\in R^0$ is regular in $R$. In this case, we necessarily have $a\in R$ since $R\subseteq Q$ is a Cuntz-Krieger subgraph, and we set $S_a'=S_a$.
	\item[(2)$\to$(1):] $w\in R^0$ is singular in $R$ and $v\in R^0$ is regular in $R$. In this case, we necessarily have $a\in R$ since $R\subseteq Q$ is a Cuntz-Krieger subgraph, and we let $S_a'$ be given as follows:
	\[\xymatrix{Y_w \ar[r]^-{\cong}& X_v \ar[rr]^-{\id_{X_v}\times\{a\}}&& X_v\times\mathsf{r}_Q^{-1}(v).}\]
	\item[(3)$\to$(1):] $w\notin R^0$ and $v\in R^0$ is regular in $R$. This case cannot occur since $R\subseteq Q$ is a Cuntz-Krieger subgraph.
	
	\item[(1)$\to$(2):] $w\in R^0$ is regular in $R$ and $v\in R^0$ is singular in $R$. In this case, there are two possibilities:
	\begin{itemize}
		\item $a\in R$. In this case, we set $S_a'=S_a\times\{a\}$.
		\item $a\notin R$. In this case, we let $S_a'$ denote the following composition:
		\[\xymatrix{Y_w \ar[r]^-{\cong}& X_v \ar[rr]^-{\id_{X_v}\times\{a\}}&& X_v\times\mathsf{r}_Q^{-1}(v).}\]
	\end{itemize} 
	\item[(2)$\to$(2):] $w\in R^0$ is singular in $R$ and $v\in R^0$ is singular in $R$. In this case, there are two possibilities:
	\begin{itemize}
		\item $a\in R$. In this case, we set $S_a'=S_a\times\id$.
		\item $a\notin R$. In this case, we let $S_a'$ denote the following composition:
		\[\xymatrix{Y_w \ar[r]^-{\cong}& X_v \ar[rr]^-{\id_{X_v}\times\{a\}}&& X_v\times\mathsf{r}_Q^{-1}(v).}\]
	\end{itemize} 
	\item[(3)$\to$(2):] $w\notin R^0$ and $v\in R^0$ is singular in $R$.  In this case, we let $S_a'$ denote the following composition:
		\[\xymatrix{Y_w \ar[r]^-{\cong}& X_v \ar[rr]^-{\id_{X_v}\times\{a\}}&& X_v\times\mathsf{r}_Q^{-1}(v).}\]
	\item[(1)$\to$(3):] $w\in R^0$ is regular in $R$ and $v\notin R^0$. In this case, we let $S_a'$ denote the following composition:
		\[\xymatrix{Y_w \ar[r]^-{\cong}& [0,1] \ar[rr]^-{\id_{[0,1]}\times\{a\}}&& [0,1]\times\mathsf{r}_Q^{-1}(v).}\]
	\item[(2)$\to$(3):] $w\in R^0$ is singular in $R$ and $v\notin R^0$. In this case, we let $S_a'$ denote the following composition:
		\[\xymatrix{Y_w \ar[r]^-{\cong}& [0,1] \ar[rr]^-{\id_{[0,1]}\times\{a\}}&& [0,1]\times\mathsf{r}_Q^{-1}(v).}\]	\item[(3)$\to$(3):] $w\notin R^0$ and $v\notin R^0$.  In this case, we let $S_a'$ denote the following composition:
		\[\xymatrix{Y_w \ar[r]^-{\cong}& [0,1] \ar[rr]^-{\id_{[0,1]}\times\{a\}}&& [0,1]\times\mathsf{r}_Q^{-1}(v).}\]\end{enumerate}
We claim that the above defines a Cuntz-Krieger $Q$-family; for this, we need to check the conditions in \autoref{spatialckfam}. Of these, all but CK1 and CK2 are immediate. CK1 for $(S',T',E')$ follows easily by construction (using that $(S,T,E)$ satisfies CK1). Finally, CK2 is checked separately in three cases, according to whether the vertex $v\in Q$ is as in (1), (2), or (3). When $v$ is as in (1), that is, regular in $R$, then CK2 for $(S',T',E')$ follows immediately from CK2 for $(S,T,E)$. When 
$v$ is as in (2), that is, singular in $R$, then CK2 does not need to be checked.
Finally, when $v$ is as in (3), that is, not in $R$, then CK2 is immediate by construction. 

It is clear from the definition that, whenever $a\in R$, then $S_a'$ is an amplification of $S_a$. This finishes the proof.
\end{proof}

\begin{cor}\label{cor:DirLim}
	Let $Q$ be a graph and let $p\in [1,\infty)$.
	Then $\mathcal{O}^p(Q)$ is the direct limit, in the category of Banach algebras with isometric homomorphism, of the $\mathcal{O}^p(R)$, for $R\in$ CK$(Q)$. Moreover, if every cycle in $Q$ has an entry, then the direct limit can be indexed over all countable Cuntz-Krieger subgraphs where every cycle has an entry.
\end{cor}
\begin{proof} 
The first assertion follows from \autoref{prop:DirLim} together with the fact, proved in \cite[Proposition 2.7]{goodearl}, that $L_Q$ is the direct limit of $L_R$, for $R\in$ CK$(Q)$. For the second assertion, it suffices to observe that by \cite[Proposition 2.7]{goodearl}, if every cycle in $Q$ has an entry, then $L_Q$ is the direct limit of $L_R$, for $R\in$ CK$(Q)$ such that every cycle has an entry. 
\end{proof}

\subsection{Uniqueness theorems}

The following lemma will allow us to assume that our countable graph has no sources, sinks, infinite emitters or infinite receivers. 
We follow some arguments in \cite[Section 8]{cortinas rodriguez}, with the necessary adjustments.

\begin{lma}\label{desingularize}
    Let $p\in[1,\infty)\setminus\{2\}$, and let $Q$ be a countable graph. Let $(X,\mathcal{B},\mu)$ be a localizable measure space, let $(S,T,E)$ be a Cuntz-Krieger $Q$-family in $\mathcal{B}(L^p(\mu))$ and let $\pi\colon\mathcal{O}^p(Q)\rightarrow \mathcal{B}(L^p(\mu))$ be the canonical contractive homomorphism. Then there exist
    \begin{enumerate}[(a)]
        \item a countable graph $Q'$ such that $Q'$ has no sources, sinks, infinite emitters or infinite receivers,

        \item a localizable measure space $(Y,\mathcal{C},\nu)$,

        \item a Cuntz-Krieger $Q'$-family $(S',T',E')$ in $\mathcal{B}(L^p(\nu))$,

        \item isometric homomorphisms $\phi\colon \mathcal{O}^p(Q)\rightarrow\mathcal{O}^p(Q')$ and $\sigma\colon \mathcal{B}(L^p(\mu))\rightarrow\mathcal{B}(L^p(\nu))$
    \end{enumerate}
    such that if $\pi'\colon\mathcal{O}^p(Q')\rightarrow\mathcal{B}(L^p(\nu))$ is the canonical homomorphism associated to $(S',T',E')$, then $\sigma\circ\pi=\pi'\circ\phi$. Moreover, if each cycle of $Q$ has an entry, then each cycle of $Q'$ has an entry.
\end{lma}

\begin{proof}
    Let $Q'$ be the graph obtained in the following way. If $v\in Q^0$ is a sink, add an infinite tail
    \begin{center}
        \begin{tikzcd}
            v\arrow[r,"f_1"] &v_1\arrow[r,"f_2"] &v_2\arrow[r,"f_3"] &\dots\rlap{\ \ .}
        \end{tikzcd}
    \end{center}
    If $v$ is a source, add an infinite head
    \begin{center}
        \begin{tikzcd}
            \dots\arrow[r,"f_3"] &v_2\arrow[r,"f_2"] &v_1\arrow[r,"f_1"] &v\rlap{\ \ .}
        \end{tikzcd}
    \end{center}
    If $v$ is an infinite emitter with $\sor^{-1}(v)=a_1,a_2,\dots$, remove each $a_n$ and replace with $g_n f_n\cdots f_1$ in the following construction
    \begin{center}
    \begin{tikzcd}
        v \arrow[r,"f_1"] & v_1\arrow[r,"f_2"]\arrow[d,"g_1"] &v_2\arrow[r,"f_3"]\arrow[d,"g_2"] &v_3\arrow[r,"f_4"]\arrow[d,"g_3"] &\dots\\
        &\ran(a_1) &\ran(a_2) &\ran(a_3) &\dots
    \end{tikzcd}
    \end{center}
    Similarly, if $v$ is an infinite receiver with $\ran^{-1}(v)=a_1,a_2,\dots$ then remove each $a_n$ and replace with $f_1\cdots f_n g_n$ as follows
    \begin{center}
        \begin{tikzcd}
            \dots\arrow[r,"f_4"] &v_3\arrow[r,"f_3"] &v_2\arrow[r,"f_2"] &v_1\arrow[r,"f_1"] &v\rlap{\ \ .}\\
            \dots &\sor(a_3)\arrow[u,"g_3"] &\sor(a_2)\arrow[u,"g_2"] &\sor(a_1)\arrow[u,"g_1"]
        \end{tikzcd}
    \end{center}
    
    Set $F=\{a\in Q^1\colon |\sor^{-1}(\sor(a))|=\infty\}$ and set $G=\{a\in Q^1\colon |\ran^{-1}(\ran(a))|=\infty\}$. We write $a_n(v)$ for elements of $F$ and $G$, or just $a_n$ if the vertex $v$ is clear. The inclusion of $Q^0\cup (Q^1\setminus(F\cup G))$ in $Q'$ together with the assignments $a_n\mapsto g_n f_n\cdots f_1$, for $a_n\in F$, and $a_n\mapsto f_1\cdots f_n g_n$, for $a_n\in G$, define an injective homomorphism $\phi\colon L_Q\rightarrow L_{Q'}$.

    We note that all cycles of $Q'$ are of the form $\phi(c)$ for some cycle $c$ in $Q$. Fixing a cycle $c$ in $Q$, we want to prove that $\phi(c)$ has an entry. Clearly adding heads and tails to sinks and sources does not affect this property. Suppose first that no vertex of $c$ is an infinite receiver. Then no edge in $G$ can be an entry for $c$. If some $a_n(v)\in F$ is an entry for $c$, then $g_n(v)$ is an entry for $\phi(c)$, and if the entry for $c$ belongs to $Q^1\setminus(F\cup G)$, then the entry is preserved by $\phi$. Now suppose that some vertex $v$ of $c$ is an infinite receiver. Then $c$ must contain some edge entering $v$, but it cannot contain all of them as $c$ is finite. Let $a_n(v)$ be an edge in $c$ such that $a_m(v)$ does not belong to $c$ for $m> n$. Then $f_{n+1}(v)$ is an entry for $c$.
    
    We want to show that the homomorphism $\phi$ extends to an isometry $\tilde{\phi}\colon\mathcal{O}^p(Q)\rightarrow\mathcal{O}^p(Q')$. In order to show this, it suffices to show that, given a spatial representation $\varphi\colon L_Q\rightarrow\mathcal{B}(L^p(\mu))$, we can find a localizable measure space $(Y,\mathcal{C},\nu)$, a spatial representation $\varphi'\colon L_{Q'}\rightarrow\mathcal{B}(L^p(\nu))$, and a spatial isometry $s\colon L^p(\mu)\rightarrow L^p(\nu)$ with reverse $t$, such that for the map $\sigma\colon \mathcal{B}(L^p(\mu))\rightarrow\mathcal{B}(L^p(\nu))$ given by $\sigma(d)= sdt$ for $d\in\mathcal{B}(L^p(\mu))$, we have $\sigma\circ\varphi=\varphi'\circ\phi$. Then, if $F$ and $F'$ are the sets of spatial representations of $L_Q$ and $L_{Q'}$, respectively, we have 
    $$\|\phi(a)\|=\sup_{\varphi'\in F'}\|\varphi'(\phi(a))\|=\sup_{\varphi\in F}\|\varphi(a)\|=\|a\|$$ 
    for all $a\in L_Q$. We set out to prove this claim. 
    
    Let $(X,\mathcal{B},\mu)$ be a localizable measure space and let $\varphi\colon L_Q\rightarrow \mathcal{B}(L^p(\mu))$ be a spatial representation. Let $H$ be the set of sources, sinks, infinite emitters and infinite receivers of $Q$. For $v\in H$, we define a measure space $Y_v$ as follows. We use the notation of \autoref{orthogonal}, so that $\varphi(s_{\alpha})$ is the spatial partial isometry associated to the spatial system $(X_{\sor(\alpha)},X_\alpha,\eta_\alpha,f_\alpha)$ for each $\alpha\in Q^*$, and we regard $\N$ as a measure space with counting measure.
    \begin{enumerate}[(i)]
        \item If $v$ is a sink or a source, set $Y_v=X_v\times\N$.

        \item If $v$ is an infinite emitter, set $X'_v=X_v\setminus\bigsqcup_{i\in\N} X_{a_i}$, set $Y_{v_n}=X'_v\sqcup \bigsqcup_{i\leq n}X_{a_i}$ and define $Y_v=\bigsqcup_{n\in \N} Y_{v_n}$.
        
        \item If $v$ is an infinite receiver, set $X'_v=X_v\setminus\bigsqcup_{i\in\N} X_{a_i}$, set $Y_{v_n}=X'_v\sqcup \bigsqcup_{i\geq n}X_{a_i}$ and define $Y_v=\bigsqcup_{n\in \N} Y_{v_n}$.
    \end{enumerate}
    We then set $Y=X\sqcup\bigsqcup_{v\in H} Y_v$ and let $\nu$ be the inherited measure on $Y$. We choose $s$ to be the spatial isometry $L^p(X)\rightarrow L^p(Y)$ induced by the inclusion $X\subseteq Y$, and we let $t$ be its reverse. We let $\varphi'\colon L_{Q'}\rightarrow \mathcal{B}(L^p(\nu))$ be the extension of $\varphi$ along $\phi$ determined as follows.
    \begin{enumerate}
        \item Suppose $v$ is a sink or a source. Let $\tau_n$ be the isometric spatial isomorphism $L^p(\nu_{X_v})\rightarrow L^p(\nu_{X_v\times \{n\}})$ induced by the identification $X_v\cong X_v\times\{n\}$. Set $\varphi'(e_{v_n})=\text{id}_{L^p(\nu_{X_v\times\{n\}})}$. If $v$ is a sink, set $\varphi'(s_{f_n})=\tau_n\tau_{n-1}^{-1}$ and set $\varphi'(t_{f_n})=\tau_{n-1}\tau_n^{-1}$. If $v$ is a source, set $\varphi'(s_{f_n})=\tau_{n-1}\tau_n^{-1}$ and set $\varphi'(t_{f_n})=\tau_n\tau_{n-1}^{-1}$.

        \item Suppose $v$ is an infinite emitter. Set $\varphi'(e_{v_n})=\text{id}_{L^p(\nu_{Y_n})}$. Let $\varphi'(s_{f_n})$ be the spatial isometry $L^p(\nu_{Y_{n-1}})\rightarrow L^p(\nu_{Y_n})$ induced by the inclusion $Y_{n-1}\subseteq Y_n$, and let $\varphi'(t_{f_n})$ be its reverse. Let $\varphi'(s_{g_n})$ be the spatial partial isometry $L^p(\nu_{Y_{v_n}})\rightarrow L^p(\nu_{X_{a_n}})$ induced by the inclusion $Y_{v_n}\subseteq X_v$ followed by the measurable set transformation $\eta_{a_n}$ from $X_v$ to $X_{a_n}$, and let $\varphi'(t_{g_n})$ be its reverse. 
        
        \item Suppose $v$ is an infinite receiver. Set $\varphi'(e_{v_n})=\text{id}_{L^p(\nu_{Y_n})}$. Let $\varphi'(s_{f_n})$ be the spatial isometry $L^p(\nu_{Y_n})\rightarrow L^p(\nu_{Y_{n-1}})$ induced by the inclusion $Y_n\subseteq Y_{n-1}$, and let $\varphi'(s_{g_n})$ be the spatial partial isometry $L^p(\nu_{X_{\sor(a_n)}})\rightarrow L^p(\nu_{Y_n})$ induced by the measurable set transformation $\eta_{a_n}$ from $X_{\sor(a_n)}$ to $X_{a_n}$ followed by the inclusion $X_{a_n}\subseteq Y_n$. Let $\varphi'(t_{f_n})$ and $\varphi'(t_{g_n})$ be the respective reverses.
    \end{enumerate}

One checks that $\varphi'$ is a well-defined spatial representation of $L_{Q'}$, and that we have $\sigma\circ\varphi=\varphi'\circ\phi$. Hence $\phi$ extends to an isometric homomorphism $\mathcal{O}^p(Q)\rightarrow\mathcal{O}^p(Q')$, which we also call $\phi$.

Now if $(S,T,E)$ is a Cuntz-Krieger $Q$-family in $\mathcal{B}(L^p(\mu))$, let $\varphi$ be the induced spatial representation of $L_Q$. Construct a graph $Q'$, a measure space $(Y,\mathcal{C},\nu)$, a spatial representation $\varphi'\colon L_{Q'}\rightarrow\mathcal{B}(L^p(\nu))$ and an isometry $\sigma\colon \mathcal{B}(L^p(\mu))\rightarrow\mathcal{B}(L^p(\nu))$ using the method detailed above. Let $(s',t',e')$ be the generating family of $L_{Q'}$ and set $(S',T',E')=(\varphi'(s'),\varphi'(t'),\varphi'(e'))$. Then the canonical homomorphism $\pi'\colon\mathcal{O}^p(Q')\rightarrow \mathcal{B}(L^p(\nu))$ satisfies $\sigma\circ\pi=\pi'\circ\phi$, which is what we wanted to prove.
\end{proof}


Now we are equipped to prove our final theorems in the desired generality. We begin by the gauge-invariant uniqueness theorem.

\begin{thm}\label{gauge-uniqueness}(Gauge-invariance uniqueness.)
    Let $p\in [1,\infty)$, let $Q$ be a graph, let $A$ be a unitizable $L^p$-operator algebra, and let $(S,T,E)$ be a Cuntz-Krieger $Q$-family in $A$ such that $E_v\neq 0$ for all $v\in Q^0$. Suppose that there is an isometric action $\beta\colon \mathbb{T}\rightarrow \text{Aut}(A)$ such that $\beta_z(S_a)=zS_a,\ \beta_z(T_a)=\overline{z}T_a$ and $\beta_z(E_v)=E_v$ for all $z\in\mathbb{T}$, for all $a\in Q^1$ and for all $v\in Q^0$. Then the canonical homomorphism $\pi\colon \mathcal{O}^p(Q)\rightarrow A$ is injective.
\end{thm}

\begin{proof}
    We first explain how to reduce the statement to the case where $Q$ is countable. By \autoref{cor:DirLim}, we have $\mathcal{O}^p(Q)\cong \varinjlim \mathcal{O}^p(R)$ for $R\in\text{CK}(Q)$. For each such $R$, it is clear that $L_Q(R)$ is invariant under the gauge action, and by continuity so is $\mathcal{O}^p(R)$. By similar reasoning, $\pi(\mathcal{O}^p(R))$ is invariant under $\beta$. Hence, for each $R\in \text{CK}(Q)$, the image of $\mathcal{O}^p(R)$ under $\pi$ is equipped with an isometric action as described above. The statement of the theorem is equivalent to the fact that if $\varphi\colon \mathcal{O}^p(Q)\to A$ is a contractive homomorphism into an $L^p$-operator algebra whose kernel intersects $C_0(Q^0)$ trivially, then $\varphi$ is injective. Let $I$ denote the kernel of $\varphi$. By the description of ideal in direct limits of Banach algebras (with isometric maps), we have $I={0}$ if and only if $I\cap \mathcal{O}^p(R)={0}$ for all $R\in$ CK$(Q)$. In other words, if we show that $\varphi$ is injective when restricted to $\mathcal{O}^p(R)$, then the result will follow.

    We will assume from now on that $Q$ is countable. By representing $A$ on the $L^p$-space of a localizable measure, we may assume that $(S,T,E)$ is a family of operators on the $L^p$-space $\mathcal{E}$ of a localizable measure. Using \autoref{desingularize}, we may assume that $Q$ has no sinks, sources, infinite emitters or infinite receivers. Suppose that 
    \begin{equation}\label{eqn:estimate}
        \|\pi(\Phi_n(a))\|\leq \|\pi(a)\|\tag{$6.1$}
    \end{equation}
    for each $n\in\Z$ and for each $a\in\mathcal{O}^p(Q)$. Then if $a\in\mathcal{O}^p(Q)$ satisfies $\pi(a)=0$ we get $\pi(\Phi_n(a))=0$ for all $n\in\Z$. By \autoref{isometryspectral}, we get $\Phi_n(a)=0$ for all $n\in\Z$, and this implies $a=0$ by part (4) of \autoref{phindef}. Thus, we set out to prove the inequality in (\ref{eqn:estimate}).

    We first prove (\ref{eqn:estimate}) for $n=0$, and then we show that the general case follows. From the description of $\beta$, it is clear that $\beta_z\circ\pi=\pi\circ\gamma_z$ for all $z\in\mathbb{T}$. Hence
    \begin{align*}
        \|\pi(\Phi_0(a))\|&\leq \int_\mathbb{T}\|\pi(\gamma_z(a))\|dz=\int_\mathbb{T}\|\beta_z(\pi(a))\|dz\\
        &=\int_\mathbb{T}\|\pi(a)\|dz=\|\pi(a)\|
    \end{align*}
    for each $a\in\mathcal{O}^p(Q)$.
    
    Now we prove that (\ref{eqn:estimate}) holds for general $n$. By continuity, it suffices to show the inequality for $a\in L_Q$. Suppose that $n\geq 0$, and fix $a=\sum_{(\alpha,\beta)\in F}\lambda_{\alpha,\beta}s_\alpha t_\beta$, where $F$ is a finite family of pairs $(\alpha,\beta)\in Q^*\times Q^*$. Set $k=\max\{|\alpha|,|\beta|\colon (\alpha,\beta)\in F\}$. Let $(\alpha,\beta)\in F$ with $\sor(\alpha)=\sor(\beta)=w$, and suppose $\min\{|\alpha|,|\beta|\}<k$. As $Q$ is source-free, we may write $s_\alpha t_\beta=\sum_{a\in \ran^{-1}(w)}s_{\alpha a}t_{\beta a}$, and repeating this process we eventually obtain a presentation of $s_\alpha t_\beta$ as a finite sum of elements of the form $s_{\alpha\gamma}t_{\beta\delta}$ with $\min\{|\alpha\gamma|,|\beta\delta|\}=k$. Thus, we may modify $F$ to ensure that $k=\min\{|\alpha|,|\beta|\}$ for all $(\alpha,\beta)\in F$. Then in particular $\Phi_n(a)\in\mathcal{F}_k^n$, where $\mathcal{F}_k^n$ is defined as in \autoref{isometryspectral}.
    
    Since $Q$ has no sinks, we may use the method in \autoref{isometryspectral} to find $x$ and $x^*$ with $\|x\|=\|x^*\|=1$ such that $\Phi_n(a)xx^*=\Phi_n(a)$ and such that $\Phi_n(a)x$ belongs to $\mathcal{O}^p(Q)^\gamma$. One checks that we have $\Phi_n(a)x=\Phi_0(ax)$. Thus, as $\|\pi\circ\Phi_0\|\leq \|\pi\|$ we obtain
    $$\|\pi(\Phi_n(a))\|=\|\pi(\Phi_n(a)x)\|=\|\pi(\Phi_0(ax))\|\leq \|\pi(ax)\|\leq\|\pi(a)\|,$$
    which is the estimate we wanted to show. The case $n<0$ is similar.
\end{proof}

We proceed to the Cuntz-Krieger uniqueness theorem. For a path $\alpha\in Q^*$, we say that $\alpha$ is \emph{nonreturning} if $\alpha_k\neq\alpha_{|\alpha|}$ for $k<|\alpha|$.

\begin{thm}\label{injectivity} (Cuntz-Krieger uniqueness.)
    Let $p\in [1,\infty)$, let $Q$ be a graph in which each cycle has an entry, let $A$ be a unitizable $L^p$-operator algebra, and let $(S,T,E)$ be a Cuntz-Krieger $Q$-family in $A$ such that $E_v\neq 0$ for all $v\in Q^0$. Then the canonical homomorphism $\pi\colon \mathcal{O}^p(Q)\rightarrow A$ is injective.
\end{thm}

\begin{proof}
    By \autoref{cor:DirLim}, if we denote by CK$_e(Q)$ the directed set of all countable Cuntz-Krieger subgraphs of $Q$ for which every cycle has an entry, then $\mathcal{O}^p(Q)\cong \varinjlim \mathcal{O}^p(R)$ for $R\in\text{CK}_e(Q)$. Hence, just as in \autoref{gauge-uniqueness}, we may assume that $Q$ is countable, and we may represent $A$ on the $L^p$-space of a localizable measure and use \autoref{desingularize} to assume that $Q$ has no sinks, sources, infinite emitters or infinite receivers. As in \autoref{gauge-uniqueness}, we prove the inequality (\ref{eqn:estimate}) for $n=0$, from which the case of $n\in\Z$ follows, and hence the statement of the theorem.

    By continuity, it suffices to show (\ref{eqn:estimate}) for $a\in L_Q$. We fix $a=\sum_{(\alpha,\beta)\in F}\lambda_{\alpha,\beta}s_\alpha t_\beta$, where $F$ is a finite family of pairs $(\alpha,\beta)\in Q^*\times Q^*$, and we shall find an operator $V\in A$ with $\|V\|\leq 1$ such that 
    \begin{enumerate}
        \item $VS_\alpha T_\beta V=0$ if $(\alpha,\beta)\in F$ and $|\alpha|\neq|\beta|$,
        
        \item $\|V\pi(\Phi_0(a)) V\|=\|\pi(\Phi_0(a))\|$.
    \end{enumerate}
    This will give 
    $$\|\pi(\Phi_0(a))\|=\|V\pi(\Phi_0(a))V\|=\|V\pi(a)V\|\leq \|\pi(a)\|,$$
    as desired.

    As in \autoref{gauge-uniqueness}, we can modify $F$ to make sure that $\Phi_0(a)\in\mathcal{F}_k$ for some $k\geq 0$. By \autoref{fkdirectsum} we have $\mathcal{F}_k=\bigoplus_{v\in Q^0}\mathcal{F}_k(v)$, so there is some $v\in Q^0$ such that for
    $$b_v=\sum_{\substack{(\alpha,\beta)\in F\\ \alpha,\beta\in Q^k\cap \sor^{-1}(v)}}\lambda_{\alpha,\beta}s_\alpha t_\beta$$
    we have $\|\Phi_0(a)\|=\|b_v\|$. We set $G=\{\alpha,\beta\colon (\alpha,\beta)\in F,\ \alpha,\beta\in Q^k\cap \sor^{-1}(v)\}$. Then $\text{span}\{s_\alpha t_\beta\colon \alpha,\beta\in G\}$ is a finite-dimensional matrix algebra containing $b_v$.

    By \cite[Lemma 3.7]{raeburn}, we may choose a nonreturning path $\lambda$ with $\ran(\lambda)=v$ and with $|\lambda|>\max\{|\alpha|,|\beta|\colon (\alpha,\beta)\in F\}$. Set $V=\sum_{\tau\in G}S_{\tau\lambda}T_{\tau\lambda}$. As the paths $\tau\lambda$ are distinct but of equal length they are incomparable, and hence using the notation of \autoref{orthogonal} the sets $X_{\tau\lambda}$, for $\tau\in G$, are disjoint. Thus $V$ is a finite sum of pairwise orthogonal spatial idempotents, and hence $\|V\|=1$.

    We observe that for $|\alpha|=k$ and for $\tau\in G$, we have $T_{\tau\lambda}S_\alpha=\delta_{\tau,\alpha}T_\lambda$, so that
    $$VS_\alpha=\mathbbm{1}_G(\alpha)S_{\alpha\lambda}T_{\alpha\lambda}S_\alpha=\mathbbm{1}_G(\alpha)S_{\alpha\lambda}T_\lambda.$$ 
    Similarily, as $T_\alpha S_{\tau\lambda}=\delta_{\alpha,\tau}S_\lambda$, we have $T_\alpha V=\mathbbm{1}_G(\alpha) S_\lambda T_{\alpha\lambda}$.

    Suppose that $(\alpha,\beta)\in F$ satisfies $|\alpha|\neq |\beta|$. Assume that $|\alpha|=k$. Then 
    $$VS_\alpha T_\beta V=\sum_{\tau\in G}\mathbbm{1}_G(\alpha)S_{\alpha\lambda}T_{\beta\lambda} S_{\tau\lambda}T_{\tau\lambda}.$$
    In order for $T_{\beta\lambda} S_{\tau\lambda}$ to be nonzero, we need that $\beta\lambda=\tau\lambda\gamma$ for some path $\gamma$ with $|\gamma|=|\beta|-|\tau|=|\beta|-k$, which is not possible as $0<|\beta|-k<|\lambda|$ and $\lambda$ is nonreturning. Hence $VS_\alpha T_\beta V=0$. The case where $|\beta|=k$ is similar, and hence $V$ satisfies (1).

    To show property (2), we first note that if $|\alpha|=|\beta|$, then
    $$VS_\alpha T_\beta V=\mathbbm{1}_G(\alpha)\mathbbm{1}_G(\beta)S_{\alpha\lambda}T_\lambda S_\lambda T_{\beta\lambda}=\mathbbm{1}_G(\alpha)\mathbbm{1}_G(\beta)S_{\alpha\lambda} T_{\beta\lambda}.$$
    By assumption, we have $E_{\sor(\tau\lambda)}\neq 0$, so $\{S_{\alpha\lambda} T_{\beta\lambda}\colon \alpha,\beta\in G\}$ is a family of nonzero matrix units. Therefore, the map
    $$\text{span}\{s_\alpha t_\beta\colon \alpha,\beta\in G\}\rightarrow \text{span}\{S_{\alpha\lambda} T_{\beta\lambda}\colon \alpha,\beta\in G\}$$
    given by $b \mapsto V\pi(b)V$
    is a nonzero contractive homomorphism out of a matrix algebra in $\mathcal{O}^p(Q)$. By \autoref{matrixiso}, the domain has the spatial norm, and thus by \autoref{afincompressible}, the homomorphism is isometric. Thus we have
    $$\|\pi(\Phi_0(a))\|=\|\Phi_0(a)\|=\|b_v\|=\|V\pi(b_v)V\|=\|V\pi(\Phi_0(a))V\|,$$
    and we have proven the second property. This finishes the proof of the inequality $\|\pi(\Phi_0(a))\|\leq \|\pi(a)\|$ for all $a\in\mathcal{O}^p(Q)$.
\end{proof}

When $p=2$, \autoref{uniqueness} holds in a stronger form, namely that the homomorphism $\pi$ is an isometry; this is an immediate consequence of the fact that injective homomorphisms between C*-algebras are automatically isometric. While the corresponding statement is false in general for $L^p$-operator algebras (as can be seen, for example, by considering the Gelfand transform $F^p(\Z)\to C(S^1)$), in some special cases we do in fact get an isometric homomorphism:

\begin{thm}\label{uniqueness}
    Let $p\in [1,\infty)$, and let $Q$ be an acyclic graph. Let $(S,T,E)$ be a Cuntz-Krieger $Q$-family in a unitizable $L^p$-operator algebra $A$ such that $E_v\neq 0$ for all $v\in Q^0$. Then the canonical homomorphism $\pi\colon\mathcal{O}^p(Q)\rightarrow A$ is an isometry.
\end{thm}

\begin{proof}
    By \autoref{acyclicaf}, $\mathcal{O}^p(Q)$ is a spatial $L^p$-AF-algebra, so by \autoref{afincompressible} it is $p$-incompressible. Thus the result follows immediately from \autoref{injectivity}.
\end{proof}

\end{document}